\documentclass[12pt]{amsart}

\usepackage{amssymb, amscd, txfonts}
\usepackage{graphicx}
\usepackage[all]{xy} 
\usepackage[mathscr]{euscript}

\numberwithin{equation}{section}

\newtheorem{theorem}{Theorem}[section]
\newtheorem{proposition}[theorem]{Proposition}
\newtheorem{lemma}[theorem]{Lemma}
\newtheorem{corollary}[theorem]{Corollary}

\newtheorem{step+}{Step}
\newtheorem{step++}{Step}

\theoremstyle{definition}

\newtheorem{example}[theorem]{Example}

\theoremstyle{remark}
\newtheorem{remark}[theorem]{Remark}
\newtheorem{claim}[theorem]{Claim}

\renewcommand{\hom}{\operatorname{Hom}}
\renewcommand{\ker}{\operatorname{Ker}}

\newcommand{\Z}{\mathbb{Z}}
\newcommand{\Q}{\mathbb{Q}}
\newcommand{\R}{\mathbb{R}}
\newcommand{\C}{\mathbb{C}}

\newcommand{\proj}{{\mathbb P}}

\newcommand{\D}{\mathcal{D}}
\newcommand{\G}{\Gamma}
\newcommand{\E}{\mathcal{E}}
\newcommand{\El}{\mathcal{E}_{\lambda}}
\newcommand{\ElI}{\mathcal{E}_{\lambda}^{I}}
\newcommand{\Ela}{\mathcal{E}_{\lambda(\alpha)}}
\newcommand{\OD}{\mathcal{O}_{\mathcal{D}}}
\newcommand{\LG}{{\rm LG}}

\newcommand{\DI}{\mathcal{D}_{I}}
\newcommand{\VI}{\mathcal{V}_{I}}
\newcommand{\HI}{\mathfrak{H}_{I_{0}}}
\newcommand{\CIo}{\mathcal{C}(I_{0})}
\newcommand{\CIov}{\mathcal{C}(I^{0})}
\newcommand{\CI}{\mathcal{C}(I)}

\newcommand{\Sp}{{\rm Sp}(2g, \mathbb{Q})}
\newcommand{\la}{\lambda(\alpha)}
\newcommand{\Ivl}{I^{0}_{\mathbb{C},\lambda}}
\newcommand{\UIZ}{U(I)_{\mathbb{Z}}}

\newcommand{\GI}{\Gamma(I)}
\newcommand{\GId}{\Gamma(I)'}
\newcommand{\GIbar}{\overline{\Gamma(I)}}
\newcommand{\GIbard}{\overline{\Gamma(I)'}}
\newcommand{\GIl}{\Gamma_{I,\ell}}
\newcommand{\GIh}{\Gamma_{I,h}}
\newcommand{\GIhd}{\Gamma_{I,h}'}
\newcommand{\XI}{\mathcal{X}(I)}
\newcommand{\XIS}{\mathcal{X}(I)^{\Sigma_{I}}}
\newcommand{\YIS}{\mathcal{Y}(I)^{\Sigma_{I}}}
\newcommand{\XIcpt}{\overline{\mathcal{X}(I)}}

\begin{document}

\title[]{Siegel operators for holomorphic differential forms}
\author[]{Shouhei Ma}
\thanks{Supported by KAKENHI 21H00971 and 20H00112} 
\address{Department~of~Mathematics, Tokyo~Institute~of~Technology, Tokyo 152-8551, Japan}
\email{ma@math.titech.ac.jp}

\begin{abstract}
We give a geometric interpretation of the Siegel operators for holomorphic differential forms on Siegel modular varieties. 
This involves extension of the differential forms over a toroidal compactification, 
and we show that the Siegel operator essentially describes 
the restriction and descent to the boundary Kuga variety via holomorphic Leray filtration. 
As a consequence, we obtain equivalence of various notions of ``vanishing at boundary'' 
for holomorphic forms. 
We also study the case of orthogonal modular varieties. 
\end{abstract}

\maketitle

\section{Introduction}\label{sec: intro}

Let $X={\D}/{\G}$ be a Siegel modular variety, 
where ${\G}$ is a neat arithmetic subgroup of ${\Sp}$ with $g>1$ 
and ${\D}$ is the Siegel upper half space of genus $g$. 
By a theorem of Weissauer \cite{We1}, $X$ has holomorphic differential forms only in degrees of the form 
$k(\alpha)=g(g+1)/2-\alpha(\alpha+1)/2$ with $0\leq \alpha \leq g$. 
More specifically, he proved that holomorphic differential forms $\omega$ of degree $k(\alpha)$ 
are the same as (or rather, reduce to) vector-valued Siegel modular forms of weight 
${\la}=((g+1)^{g-\alpha}, (g-\alpha)^{\alpha})$. 

The boundary behavior of $\omega$ is somewhat different between the case $\alpha=0$ and the case $\alpha>0$. 
When $\alpha=0$, namely $\omega$ is a canonical form, it has logarithmic pole at the boundary in general, 
and extendability as a differential form is equivalent to cuspidality as a modular form. 
In \cite{Ma2}, the Siegel operators for canonical forms were interpreted as the residue maps, 
and hence related with the mixed Hodge structure of $X$. 
On the other hand, when $\alpha>0$, Freitag-Pommerening \cite{FP} proved that 
$\omega$ always extends holomorphically over a smooth compactification of $X$. 
Accordingly, the corresponding weight filtration is trivial (\cite{Ma2}). 
The purpose of this paper is nevertheless to investigate a geometric interpretation of 
the Siegel operators for holomorphic differential forms in the case $\alpha>0$. 
It turns out that this still has applications to the cohomology of $X$. 

To be more precise, let $k=k(\alpha)$ with $0< \alpha <g$. 
Let $X_I$ be a cusp of corank $r$ in the Satake compactification $X^{\ast}$ of $X$, 
which corresponds to an isotropic subspace $I$ of ${\Q}^{2g}$ of dimension $r$ up to ${\G}$-equivalence. 
The cusp $X_I$ itself is a Siegel modular variety of genus $g-r$. 
Let $\Phi_{I}$ be the Siegel operator at $X_I$. 
As proved by Weissauer \cite{We1}, 
we have $\Phi_{I}(\omega)=0$ when $r>\alpha$, 
while when $r\leq \alpha$, $\Phi_{I}(\omega)$ is a vector-valued Siegel modular form of weight 
${\la}'=((g+1)^{g-\alpha}, (g-\alpha)^{\alpha-r})$ on $X_I$. 
In what follows, we assume $r\leq \alpha$. 

We choose a smooth projective toroidal compactification $X^{\Sigma}$ of $X$ with simple normal crossing boundary divisor. 
Let $\pi\colon X^{\Sigma}\to X^{\ast}$ be the natural morphism. 
Each irreducible component of the $I$-locus $\pi^{-1}(X_I)\subset X^{\Sigma}$, 
labelled by a suitable cone $\sigma$ and denoted as $\overline{\Delta([ \sigma ])}^{I}$, has the structure 
\begin{equation*}
\overline{\Delta([ \sigma ])}^{I} \stackrel{\pi_{\sigma}}{\longrightarrow } Y_I \stackrel{\pi_I}{\longrightarrow} X_{I}',  
\end{equation*}
where 
$X_{I}'$ is a finite cover of $X_{I}$, 
$\pi_{I}$ is a smooth abelian fibration, and 
$\pi_{\sigma}$ is a smooth projective toric fibration. 
Since the $\pi_{\sigma}$-fibers have no holomorphic forms of positive degree,  
the restriction of $\omega$ to $\overline{\Delta([ \sigma ]) }^{I}$ as a differential form 
is the pullback of a holomorphic $k$-form $\omega_{I}$ on $Y_I$. 
Although the projections $\pi_{\sigma}$ are \textit{not} canonical 
and hence cannot be glued over $\pi^{-1}(X_{I})$ unless $r=1$ (see \S \ref{ssec: full toroidal}), 
we show that $\omega_{I}$ does not depend on $\sigma$. 

Our result says to the effect that $\Phi_{I}(\omega)$ is the descent of $\omega_I$ from $Y_I$ to $X_{I}'$. 
Although the vector bundle $\Omega_{Y_{I}}^{k}$ itself does not descend to $X_{I}'$ unless $r=\alpha$, 
we can do so after breaking it up by the holomorphic Leray filtration 
$L^{\bullet}\Omega_{Y_{I}}^{k} = \pi_{I}^{\ast}\Omega_{X_{I}'}^{\bullet}\wedge \Omega_{Y_{I}}^{k-\bullet}$. 
We have 
\begin{equation*}\label{eqn: Gr hol Leray intro}
{\rm Gr}_{L}^{l}\Omega_{Y_{I}}^{k} \simeq \pi_{I}^{\ast}(\Omega_{X_{I}'}^{l}\otimes \pi_{I \ast}\Omega_{\pi_{I}}^{k-l}) 
\end{equation*}
(see Lemma \ref{lemma: GrL}). 
We denote by ${\E}_{{\la}'}$ the automorphic vector bundle of weight ${\la}'$ on $X_{I}'$. 
Then we can regard $\Phi_{I}(\omega)$ as a section of ${\E}_{{\la}'}$ over $X_{I}'$. 
We can now state our result. 

\begin{theorem}\label{thm: main}
Let $k=k(\alpha)$ with $r \leq \alpha < g$. 
We put $l=k-r(g-\alpha)$. 
There exists an embedding 
$\pi_{I}^{\ast}{\E}_{{\la}'}\hookrightarrow L^{l}\Omega_{Y_{I}}^{k}$
(independent of $\omega$) such that 
the $k$-form $\omega_{I}$ takes values in $\pi_{I}^{\ast}{\E}_{{\la}'}$ and satisfies 
\begin{equation*}\label{eqn: main}
\omega_{I} = \pi_{I}^{\ast} \Phi_I(\omega) 
\end{equation*}
as sections of $\pi_{I}^{\ast}{\E}_{{\la}'}$. 
The projection 
$\pi_{I}^{\ast}{\E}_{{\la}'}\to {\rm Gr}_{L}^{l}\Omega_{Y_{I}}^{k}$ 
is injective and descends to 
${\E}_{{\la}'}\hookrightarrow \Omega_{X_{I}'}^{l}\otimes \pi_{I \ast}\Omega_{\pi_{I}}^{r(g-\alpha)}$ 
over $X_{I}'$.  
\end{theorem}

Interpreted more geometrically, this says that $\omega_{I}$ has pure Leray level $k-r(g-\alpha)$ and 
$\Phi_I(\omega)$ is the descent of $\omega_{I}$ as a section of ${\rm Gr}_{L}^{l}\Omega_{Y_{I}}^{k}$. 
Note that the Leray level $l=k-r(g-\alpha)$ can also be written as 
\begin{equation*}
l = (g-r)(g-r+1)/2 - (\alpha-r)(\alpha-r+1)/2, 
\end{equation*}
namely the Weissauer degree of corank $\alpha-r$ in genus $g-r$. 

When $r=\alpha$, we have $k=\dim Y_{I}$. 
Then $K_{Y_I}\simeq \pi_{I}^{\ast}(K_{X_{I}'}\otimes \pi_{I\ast}K_{\pi_{I}})$, 
and $K_{X_{I}'}\otimes \pi_{I\ast}K_{\pi_{I}}$ is the line bundle of scalar-valued modular forms of weight $(g-r+1)+r=g+1$ on $X_{I}'$. 
In this case, Theorem \ref{thm: main} simply says that 
$\Phi_{I}(\omega)$ is the cusp form of weight $g+1$ 
corresponding to $\omega_{I}$ by this isomorphism. 
When $r=\alpha =1$, this was essentially observed by Weissauer (\cite{We2} p.309). 
Theorem \ref{thm: main} is a generalization of his observation. 
Holomorphic Leray filtration arises in the case $r<\alpha$, 
via which the descent to $X_{I}'$ can be understood more geometrically. 

In the course of the proof of Theorem \ref{thm: main}, 
we also obtain a new proof of the Freitag-Pommerening extension theorem (Remark \ref{remark: FP new proof}). 
In fact, the present study grew out of an effort to understand the theorem of Freitag-Pommerening, 
and can be regarded as a refinement of 
the intersection of the Freitag-Pommerening theorem and Weissauer's theory. 


A natural consequence of Theorem \ref{thm: main} is 
the following equivalence of various notions of ``vanishing at boundary" for $\omega$:  

\begin{corollary}\label{cor: intro}
The following conditions for $\omega$ are equivalent:

(1) As a modular form, $\omega$ is a cusp form. 

(2) As a differential form, the restriction of $\omega$ to every irreducible component of $X^{\Sigma}-X$ vanishes. 

(3) As a cohomology class, the restriction of $\omega$ to the Borel-Serre boundary vanishes. 
\end{corollary}

Here Theorem \ref{thm: main} is used to deduce $(1) \Leftrightarrow (2)$ (Corollary \ref{cor: cuspidality}), 
while $(2) \Leftrightarrow (3)$ follows from a general argument in mixed Hodge theory (Corollary \ref{cor: cuspidality Borel-Serre}). 


For the proof of Theorem \ref{thm: main}, 
we need to lift the Siegel operator to the level of toroidal compactifications. 
This is done in \S \ref{sec: Siegel} in the full generality of vector-valued modular forms of arbitrary weight. 
Theorem \ref{thm: main} is deduced in \S \ref{sec: Omega} 
by combining the result of \S \ref{sec: Siegel} for the weight $\lambda(\alpha)$ and 
consideration of filtrations arising from the Siegel domain realization. 
Some applications to the cohomology of $X$ are given in \S \ref{sec: corank}. 
In \S \ref{sec: orthogonal}, we prove a similar (and simpler) result for orthogonal modular varieties. 

The following notation will be used throughout this paper. 
For a $K$-linear space $V$ of finite dimension, 
its dual space $\hom_{K}(V, K)$ is denoted by $V^{\vee}$. 
For a field extension $K'/K$, we write $V_{K'}=V\otimes_{K}K'$. 
We denote by $S^{2}V={\rm Sym}^{2}V$ the symmetric square of $V$. 
When $K\subset {\R}$, we write $\mathcal{C}(V)=\mathcal{C}(V_{{\R}})$ for the cone of positive-definite forms in $S^{2}V_{{\R}}$.

\section{Toroidal compactifications and automorphic vector bundles}\label{sec: toroidal}

In this section we recall the basic theory of 
toroidal compactifications of Siegel modular varieties and vector-valued Siegel modular forms, 
following more or less \cite{AMRT}, \cite{Mu}, \cite{HZ1}, \cite{vdG} 
but in a unified and somewhat self-contained manner.

\subsection{Siegel modular varieties}\label{ssec: modular variety}

Let $g>1$. 
We equip ${\Q}^{2g}$ with the standard symplectic form $(\cdot , \cdot )$. 
Let ${\rm LG}(g)$ be the Lagrangian Grassmannian parametrizing maximal isotropic subspaces $V$ of ${\C}^{2g}$. 
We define the Hermitian symmetric domain ${\D}$ as the open set of ${\rm LG}(g)$ parametrizing those $[V]$ such that  
the Hermitian form $\sqrt{-1}(\cdot , \bar{\cdot})|_{V}$ on $V$ is positive-definite. 
We have $\dim {\D} = g(g+1)/2$. 

Rational boundary components (cusps) of ${\D}$ correspond to isotropic subspaces of ${\Q}^{2g}$. 
If $I\subset {\Q}^{2g}$ is an isotropic subspace, 
the corresponding cusp ${\DI}$ 
can be identified with the Hermitian symmetric domain for the symplectic space $V(I)=I^{\perp}/I$. 
In particular, a maximal isotropic subspace $I_{0}$ corresponds to  
the $0$-dimensional cusp $[I_{0,{\C}}]\in {\rm LG}(g)$.  
If we choose a complementary isotropic subspace $I_{0}'\subset {\Q}^{2g}$, 
the graph construction with respect to $I_{0}, I_{0}'$ gives a birational map 
${\rm LG}(g)\dashrightarrow S^{2}I_{0,{\C}}$. 
This realizes ${\D}$ as the Siegel upper half space 
\begin{equation*}
{\HI}=\{ \: Z\in S^{2}I_{0,{\C}}\: | \: {\rm Im} (Z) \in {\CIo} \: \}.  
\end{equation*} 

Let ${\G}$ be an arithmetic subgroup of ${\rm Sp}(2g, {\Q})$. 
We assume that ${\G}$ is neat. 
The quotient $X={\D}/{\G}$ has the structure of a smooth quasi-projective variety, 
embedded as a Zariski open set in the Satake compactification 
$X^{\ast} = {\D}^{\ast}/{\G}$ where ${\D}^{\ast}$ is the union of ${\D}$ and its cusps equipped with the Satake topology. 
We denote by $X_I$ the $I$-stratum in $X^{\ast}$,  
and also call it the \textit{$I$-cusp}.  
We call $\dim I$ the \textit{corank} of $X_{I}$.

\subsection{Automorphic vector bundles}\label{ssec: automorphic VB}

Let ${\E}$ be the \textit{Hodge bundle} on ${\D}$, 
namely the universal sub vector bundle whose fiber over $[V]\in {\D}$ is $V$ itself. 
Let $\lambda=(\lambda_1 \geq \cdots \geq \lambda_{g}\geq 0)$ be a partition of length $\leq g$ 
(a highest weight for ${\rm GL}(g, {\C})$). 
The Schur functor associated to $\lambda$ produces from ${\E}$ an equivariant vector bundle ${\El}$ on ${\D}$. 
This is called the \textit{automorphic vector bundle} of weight $\lambda$.  
A ${\G}$-invariant section of ${\El}$ over ${\D}$ is called a \textit{modular form} of weight $\lambda$ with respect to ${\G}$. 
We denote by $M_{\lambda}({\G})$ the space of such modular forms. 

We choose a maximal isotropic subspace $I_0$ of ${\Q}^{2g}$ 
and write $I^{0}=I_{0}^{\vee}$. 
Since the pairing between $I_{0,{\C}}$ and $V$ is nondegenerate for $[V]\in {\D}$, 
this defines an isomorphism ${\E}\to I^{0}_{{\C}}\otimes {\OD}$. 
This induces an isomorphism ${\El}\to {\Ivl}\otimes {\OD}$ for every $\lambda$, 
where ${\Ivl}$ is the space obtained by applying the $\lambda$-Schur functor to $I^{0}_{{\C}}$. 
We call this isomorphism the \textit{$I_{0}$-trivialization} of ${\El}$. 

\begin{example}\label{ex: I-trivialization TD}
The tube domain realization ${\D}\simeq \frak{H}_{I_{0}}\subset S^{2}I_{0,{\C}}$ 
at $I_0$ induces an isomorphism 
$T_{{\D}} \simeq S^{2}I_{0,{\C}}\otimes {\OD}$. 
This coincides with the $I_0$-trivialization of $T_{{\D}}$. 
\end{example}

Via the $I_{0}$-trivialization and the tube domain realization at $I_{0}$, 
a ${\G}$-modular form of weight $\lambda$ is identified with a ${\Ivl}$-valued holomorphic function on ${\HI}$. 
It is invariant under the translation by a full lattice $U(I_{0})_{{\Z}}$ in $S^2I_{0}$, 
and so admits a Fourier expansion of the form 
\begin{equation}\label{eqn: Fourier expansion}
f(Z) = \sum_{l\in U(I_0)_{{\Z}}^{\vee}} a(l)q^{l}, \qquad Z\in {\HI}, 
\end{equation}
where $a(l)\in {\Ivl}$, $q^{l}=\exp(2\pi i(l, Z))$, 
and $U(I_0)_{{\Z}}^{\vee}\subset S^{2}I^{0}$ is the dual lattice of $U(I_0)_{{\Z}}$. 
The Koecher principle says that $a(l)\ne 0$ only when $l\in \overline{{\CIov}}$.  
The modular form $f$ is called a \textit{cusp form} if $a(l)\ne 0$ only when $l\in {\CIov}$ 
at every $0$-dimensional cusp $I_0$. 
We denote by $S_{\lambda}({\G})\subset M_{\lambda}({\G})$ the subspace of cusp forms. 

When $I_{0}$ is the standard isotropic subspace ${\Q}^{g}\subset {\Q}^{2g}$, 
this recovers the traditional definition of Siegel modular forms. 
Since we prefer to treat all $0$-dimensional cusps equally, 
we do not normalize $I_0$ to the standard one.

\subsection{Siegel domain realization}\label{ssec: Siegel domain}

Let $I$ be an isotropic subspace of ${\Q}^{2g}$ of dimension $r$. 
Let ${\LG}(V(I))\simeq {\LG}(g-r)$ be the Lagrangian Grassmannian for $V(I)=I^{\perp}/I$. 
We denote by ${\LG}^{\circ}(I^{\perp})$ the parameter space of $(g-r)$-dimensional isotropic subspaces $W$ 
of $I^{\perp}_{{\C}}$ satisfying $W\cap I_{{\C}}= \{ 0 \}$. 
We have the natural projections 
\begin{equation*}\label{eqn: Siegel domain LG}
{\LG}(g) \dashrightarrow {\LG}^{\circ}(I^{\perp}) \dashrightarrow {\LG}(V(I)), \quad 
V\mapsto W=V\cap I^{\perp}_{{\C}}\mapsto [W], 
\end{equation*}
where $[W]$ is the image of $W\subset I^{\perp}_{{\C}}$ in $V(I)_{{\C}}$. 
Restricting this to open subsets gives the extended two-step fibration 
\begin{equation}\label{eqn: Siegel domain}
{\D}\subset {\D}(I) \stackrel{\pi_{1}}{\to} {\VI} \stackrel{\pi_{2}}{\to} {\DI}, 
\end{equation}
where 
${\VI}$ is an affine space bundle over ${\DI}$ isomorphic to $I_{{\C}}\otimes {\E}'$ 
with ${\E}'$ the Hodge bundle on ${\DI}$, 
${\D}(I)$ is a principal $S^2I_{{\C}}$-bundle over ${\VI}$,  
and inside which ${\D}$ is a fibration of Siegel upper half spaces $\frak{H}_{I}$ over ${\VI}$. 
This is the Siegel domain realization of ${\D}$ with respect to $I$. 
If we choose a maximal isotropic subspace $I_0\subset {\Q}^{2g}$ containing $I$, 
the tube domain realization at $I_{0}$ identifies \eqref{eqn: Siegel domain} with a restriction of the linear projection 
\begin{equation*}
S^{2}I_{0,{\C}} \to S^{2}I_{0,{\C}}/S^{2}I_{{\C}} \to S^{2}(I_{0}/I)_{{\C}} 
\end{equation*}
to open subsets. 

Let $P(I)$ be the stabilizer of $I$ in ${\Sp}$. 
We define a filtration 
\begin{equation}\label{eqn: filtration Q}
U(I)\lhd W(I) \lhd P(I)' \lhd P(I) 
\end{equation}
by 
\begin{equation*}
P(I)' = \ker (P(I) \to {\rm GL}(I)), 
\end{equation*}
\begin{equation*}\label{eqn: W(I)}
W(I) = \ker (P(I) \to {\rm GL}(I)\times {\rm Sp}(V(I))), 
\end{equation*}
\begin{equation*}\label{eqn: U(I)}
U(I) = \ker (P(I) \to {\rm GL}(I^{\perp})). 
\end{equation*}
Then $W(I)$ is the unipotent radical of $P(I)$ and $U(I)$ is the center of $W(I)$. 
More specifically, we have a canonical isomorphism $U(I)\simeq S^2I$, and 
$W(I)$ has the structure of a Heisenberg group: 
\begin{equation*}
0 \to U(I) \to W(I) \to I\otimes V(I) \to 0. 
\end{equation*}
For $K={\R}, {\C}$, we define $P(I)_{K}, W(I)_{K}, U(I)_{K}$ similarly. 
Then ${\D}(I)\to {\VI}$ is a principal $U(I)_{{\C}}$-bundle, 
and $W(I)_{{\R}}/U(I)_{{\R}}\simeq I_{{\R}}\otimes V(I)_{{\R}}$ 
acts on each fiber of ${\VI}\to {\DI}$ freely and transitively as translation. 

We take the intersection of \eqref{eqn: filtration Q} with ${\G}$ and denote it by 
\begin{equation}\label{eqn: filtration Z}
{\UIZ} \lhd W(I)_{{\Z}} \lhd {\GId} \lhd {\GI}. 
\end{equation}
Then we take the quotient of \eqref{eqn: filtration Z} by ${\UIZ}$ and denote it by 
\begin{equation*}
\{ 0 \} \lhd \overline{W(I)}_{{\Z}} \lhd {\GIbard} \lhd {\GIbar}. 
\end{equation*}
In particular, $\overline{W(I)}_{{\Z}}$ is a full lattice in $I\otimes V(I)$. 
We will also consider the following (neat) arithmetic groups:  
\begin{equation*}
{\GIh} = {\rm Im}({\GI} \to {\rm Sp}(V(I))), \quad 
{\GIhd} = {\rm Im}({\GId} \to {\rm Sp}(V(I))), 
\end{equation*}
\begin{equation*}
{\GIl} = {\rm Im}({\GI} \to {\rm GL}(I)). 
\end{equation*}
If we denote by ${\G}_{I}$ the image of ${\GI}\to {\rm GL}(I)\times {\rm Sp}(V(I))$, 
then ${\GIhd}$ is the intersection ${\G}_{I}\cap {\rm Sp}(V(I))$, 
while ${\GIh}$ is the image of the projection ${\G}_{I}\to {\rm Sp}(V(I))$. 
This convinces us that ${\GIh}\ne {\GIhd}$ in general. 

We choose a $0$-dimensional cusp $I_0\supset I$. 
Let $j(\gamma, x)$ be the factor of automorphy for the ${\rm Sp}(2g, {\R})$-action on ${\El}$ with respect to the $I_0$-trivialization. 
This is the ${\rm GL}({\Ivl})$-valued function on ${\rm Sp}(2g, {\R})\times {\D}$ defined by the composition 
\begin{equation*}
{\Ivl} \to ({\El})_{x} \stackrel{\gamma}{\to} ({\El})_{\gamma x} \to {\Ivl},  
\end{equation*}
where the first and the last maps are the $I_{0}$-trivialization at $x$ and $\gamma x$ respectively. 
When $\gamma$ stabilizes $I_{0,{\R}}$, this coincides with the $\gamma$-action on ${\Ivl}$. 
The following property will be used later. 

\begin{lemma}\label{lem: f.a. constant} 
If $\pi_{1}(x)=\pi_{1}(x')$, then $j(\gamma, x)=j(\gamma, x')$ for every $\gamma\in P(I)_{{\R}}$. 
\end{lemma}

\begin{proof}
The vector bundle ${\El}$ is naturally defined over ${\D}(I)\subset {\LG}(g)$ 
and is equivariant with respect to $U(I)_{{\C}}\cdot P(I)_{{\R}}$. 
We may write $x'=\gamma_0x$ for some $\gamma_0\in U(I)_{{\C}}$. 
Since $I_{0}\subset I^{\perp}$, $U(I)_{{\C}}$ acts on $I_0$ trivially, so we have 
$j(\gamma_{0}, \cdot) \equiv {\rm id}$ for any $\gamma_{0}\in U(I)_{{\C}}$. 
Using this property, we can calculate 
\begin{eqnarray*}
j(\gamma, \gamma_0x) 
& = & j(\gamma \gamma_0, x) \circ j(\gamma_0, x)^{-1} \; = \; j(\gamma\gamma_0, x) \\ 
& = & j(\gamma \gamma_0 \gamma^{-1}, \gamma x)\circ j(\gamma, x) \; = \; j(\gamma, x). 
\end{eqnarray*}
In the last equality we used 
$\gamma \gamma_0 \gamma^{-1} \in U(I)_{{\C}}$. 
\end{proof}

\subsection{Partial toroidal compactification}\label{ssec: partial toroidal}

We take the quotients 
\begin{equation*}
{\XI}={\D}/{\UIZ}, \quad \mathcal{T}(I)={\D}(I)/{\UIZ}. 
\end{equation*}
Let $T(I)$ be the algebraic torus $U(I)_{{\C}}/{\UIZ}$. 
Then $\mathcal{T}(I)$ is a principal $T(I)$-bundle over ${\VI}$. 
Let ${\CI}^{\ast}\subset S^2I_{{\R}}$ be the cone of positive-semidefinite forms with rational kernel. 
This is the union of ${\CI}$ with the boundary components $\mathcal{C}(J)$ for all $J\subset I$ (including $J=\{ 0 \}$ by convention). 
We take a ${\GIl}$-admissible rational polyhedral cone decomposition $\Sigma_{I}$ of ${\CI}^{\ast}$. 
This defines a torus embedding $T(I)\hookrightarrow T(I)^{\Sigma_{I}}$ 
and hence a relative torus embedding $\mathcal{T}(I)\hookrightarrow \mathcal{T}(I)^{\Sigma_{I}}$. 
Then let ${\XI}^{\Sigma_{I}}$ be the interior of the closure of ${\XI}$ in $\mathcal{T}(I)^{\Sigma_{I}}$. 
This is the same as taking relatively the interior of the closure of each fiber of ${\XI}\to {\VI}$ in $T(I)^{\Sigma_{I}}$. 
Then \eqref{eqn: Siegel domain} descends and extends to the two-step fibration 
\begin{equation}\label{eqn: Siegel domain extended}
{\XIS} \to {\VI} \to {\DI}. 
\end{equation}
When $\Sigma_{I}$ is \textit{regular}, namely every cone $\sigma\in \Sigma_{I}$ is generated by a part of a ${\Z}$-basis of ${\UIZ}$, 
then ${\XI}^{\Sigma_{I}}$ is nonsingular. 
In what follows, we always assume this. 

By construction, the boundary $\Delta(I)_{\mathcal{X}}'={\XIS}-{\XI} $ of ${\XIS}$ is naturally stratified, 
whose strata are labelled by the cones in $\Sigma_{I}$. 
For $\sigma\in \Sigma_{I}$, we denote by $\Delta(I, \sigma)_{\mathcal{X}}$ the corresponding stratum. 
If the interior of $\sigma$ is contained in ${\CI}$ (which is equivalent to $\sigma\cap {\CI} \ne \emptyset$), 
then $\Delta(I, \sigma)_{\mathcal{X}}$ is a principal torus bundle over ${\VI}$ 
for the quotient torus of $T(I)$ defined by $\sigma$. 
We call such a cone $\sigma$ an \textit{$I$-cone}, 
and denote by $\Sigma_{I}^{\circ}\subset \Sigma_{I}$ 
the set of all $I$-cones.  
Then we write 
\begin{equation}\label{eqn: D(I)X}
\Delta(I)_{\mathcal{X}} = \bigsqcup_{\sigma\in \Sigma_{I}^{\circ}} \Delta(I, \sigma)_{\mathcal{X}}. 
\end{equation}
This is closed in $\Delta(I)_{\mathcal{X}}'$. 
The irreducible decomposition of $\Delta(I)_{\mathcal{X}}$ is given by 
\begin{equation*}
\Delta(I)_{\mathcal{X}} = \sum_{\sigma} \overline{\Delta(I, \sigma)}_{\mathcal{X}}, 
\end{equation*}
where now $\sigma$ ranges over $I$-cones which are minimal in $\Sigma_{I}^{\circ}$ and 
\begin{equation}\label{eqn: irr comp D(I)X}
\overline{\Delta(I, \sigma)}_{\mathcal{X}} = 
\bigsqcup_{\sigma' \succeq \sigma} \Delta(I, \sigma')_{\mathcal{X}} 
\end{equation}
is the closure of $\Delta(I, \sigma)_{\mathcal{X}}$ in $\Delta(I)_{\mathcal{X}}$. 
Then $\overline{\Delta(I, \sigma)}_{\mathcal{X}}$ is a smooth projective toric fibration over ${\VI}$. 

On the other hand, if $\sigma\in \Sigma_{I}$ is contained in the boundary ${\CI}^{\ast}-{\CI}$, 
let $J$ be the subspace of $I$ such that $\sigma$ is a $J$-cone. 
Since $U(J)_{{\Z}}\subset {\UIZ}$, we have a natural map $\mathcal{X}(J)\to {\XI}$. 
This extends to $\mathcal{X}(J)^{\Sigma_{J}}\to {\XI}^{\Sigma_{I}}$ where 
$\Sigma_{J}$ is the restriction of $\Sigma_{I}$ to $\mathcal{C}(J)^{\ast}\subset \mathcal{C}(I)^{\ast}$. 
Then the $\sigma$-stratum $\Delta(I, \sigma)_{\mathcal{X}}$ in ${\XI}^{\Sigma_{I}}$ 
is the image of the $\sigma$-stratum $\Delta(J, \sigma)_{\mathcal{X}}$ in $\mathcal{X}(J)^{\Sigma_{J}}$ 
(see \cite{AMRT} p.164). 

Since ${\GIl}$ preserves $\Sigma_{I}$, the group ${\GIbar}$ acts on ${\XIS}$. 
It will be useful to consider the intermediate quotient 
\begin{equation*}
{\YIS} = {\XIS}/{\GIbard}. 
\end{equation*}
If we put 
\begin{equation*}
Y_{I} = {\VI}/{\GIbard}, \quad X_{I}' = {\DI}/{\GIhd}, 
\end{equation*}
then \eqref{eqn: Siegel domain extended} descends to 
\begin{equation}\label{eqn: Siegel domain Y}
{\YIS} \to Y_I \stackrel{\pi_{I}}{\to} X_{I}'. 
\end{equation}
Here ${\YIS} \to Y_I$ is a family of open subsets of $T(I)^{\Sigma_{I}}$, 
$Y_I \to X_{I}'$ is a smooth abelian fibration with translation lattice $\overline{W(I)}_{{\Z}}$, 
and $X_{I}'$ is a finite cover of $X_{I}={\DI}/{\GIh}$. 
(The difference of $X_{I}'$ and $X_{I}$ seems to have been overlooked in some literatures.) 
The group ${\GIbard}$ preserves each boundary stratum $\Delta(I, \sigma)_{\mathcal{X}}$ of ${\XIS}$. 
We denote by 
\begin{equation*}
\Delta(I, \sigma)_{\mathcal{Y}} = \Delta(I, \sigma)_{\mathcal{X}}/{\GIbard} 
\end{equation*}
the corresponding boundary stratum in ${\YIS}$. 
When $\sigma$ is an $I$-cone, $\Delta(I, \sigma)_{\mathcal{Y}}$ is a principal torus bundle over $Y_{I}$. 
For such $\sigma$, we denote by 
\begin{equation*}
\overline{\Delta(I, \sigma)}_{\mathcal{Y}} 
= \overline{\Delta(I, \sigma)}_{\mathcal{X}}/{\GIbard} 
= \bigsqcup_{\sigma'\succeq \sigma} \Delta(I, \sigma')_{\mathcal{Y}} 
\end{equation*}
the closure of $\Delta(I, \sigma)_{\mathcal{Y}}$. 
This is a smooth projective toric fibration over $Y_{I}$.

The remaining group ${\GIbar}/{\GIbard}\simeq {\GIl}$ acts on ${\YIS}$. 
It permutes the strata $\Delta(I, \sigma)_{\mathcal{Y}}$ according to the action of ${\GIl}$ on $\Sigma_{I}$. 
This is free on $\Sigma_{I}^{\circ}$. 
Note that ${\GIbar}/{\GIbard}$ also acts on $\pi_{I} \colon Y_{I}\to X_{I}'$ nontrivially. 
The action on $X_{I}'$ is through the finite quotient ${\GIh}/{\GIhd}$. 
The kernel of 
\begin{equation*}
{\GIl} \simeq {\GIbar}/{\GIbard} \to {\GIh}/{\GIhd} 
\end{equation*}
is the subgroup $\Gamma_{I}\cap {\rm GL}(I)$ of ${\GIl}$. 
This acts on the $\pi_{I}$-fibers, whose action on $H_1$ comes from the ${\rm GL}(I)$-action on $I\otimes V(I)$. 

\subsection{Canonical extension}\label{ssec: canonical extension}

Let ${\El}$ be an automorphic vector bundle on ${\D}$. 
We explain the canonical extension of ${\El}$ over ${\XI}^{\Sigma_{I}}$ (\cite{Mu}) 
in an explicit form which is suitable for dealing with Fourier expansion. 
We use the same symbol ${\El}$ for the descent of ${\El}$ to ${\XI}$. 
We choose a $0$-dimensional cusp $I_0\supset I$. 
Since ${\UIZ}$ acts on $I_0$ trivially, 
the $I_{0}$-trivialization of ${\El}$ over ${\D}$ descends to an isomorphism 
${\El}\simeq {\Ivl}\otimes \mathcal{O}_{{\XI}}$ over ${\XI}$. 
Then we extend ${\El}$ to a vector bundle over ${\XI}^{\Sigma_{I}}$ (still denoted by ${\El}$) so that 
this isomorphism extends to ${\El}\simeq {\Ivl}\otimes \mathcal{O}_{{\XI}^{\Sigma_{I}}}$. 
This is called the \textit{canonical extension} of ${\El}$ over ${\XI}^{\Sigma_{I}}$. 
Lemma \ref{lem: f.a. constant} ensures that this is in fact independent of $I_0$. 

\begin{example}
The canonical extension of $\Omega_{{\XI}}^{k}$ is 
$\Omega_{{\XIS}}^{k}(\log \Delta(I)_{\mathcal{X}}')$,  
the locally free sheaf of $k$-forms with at most logarithmic singularity at the boundary (\cite{Mu} Proposition 3.4). 
\end{example}

The vector bundle ${\El}$ descends to ${\VI}$, 
and the canonical extension is its pullback by ${\XIS}\to {\VI}$ 
(cf.~\cite{Mu} p.257 and \cite{HZ1} \S 3.2):   

\begin{lemma}\label{lemma: descend to VI}
There exists a $P(I)_{{\R}}$-equivariant vector bundle ${\El}|_{{\VI}}$ on ${\VI}$ 
with the following properties. 

(1) We have ${\El}\simeq \pi^{\ast}_{1}({\El}|_{{\VI}})$ as $P(I)_{{\R}}$-equivariant vector bundles on ${\D}$. 

(2) The $I_{0}$-trivialization of ${\El}$ descends to an isomorphism ${\Ivl}\otimes \mathcal{O}_{{\VI}}\simeq {\El}|_{{\VI}}$. 

(3) The isomorphism in (1) extends to an isomorphism between the canonical extension of ${\El}$ 
and the pullback of ${\El}|_{{\VI}}$ by ${\XIS}\to {\VI}$. 
\end{lemma}

\begin{proof}
Let $j(\gamma, x)$ be the factor of automorphy for the $P(I)_{{\R}}$-action on ${\El}$ with respect to the $I_{0}$-trivialization. 
By Lemma \ref{lem: f.a. constant}, 
$j(\gamma, x)$ descends to a function on $P(I)_{{\R}}\times {\VI}$ 
(trivial on $U(I)_{{\R}}\times {\VI}$). 
This defines the factor of automorphy of a $P(I)_{{\R}}$-equivariant vector bundle on ${\VI}$: 
this is ${\El}|_{{\VI}}$. 
By construction it is equipped with a trivialization ${\El}|_{{\VI}}\simeq {\Ivl}\otimes \mathcal{O}_{{\VI}}$. 
Then (1) and (2) follow from the construction. 
The assertion (3) holds because the canonical extension is also defined via the $I_{0}$-trivialization. 
\end{proof}

In the descent to ${\El}|_{{\VI}}$, 
the fibers of ${\El}$ over points in the same $\pi_{1}$-fiber are identified via the $I_{0}$-trivialization (by construction). 
Since $U(I)_{{\C}}$ acts on $I_{0, {\C}}$ trivially, 
this is the same as the identification via the equivariant $U(I)_{{\C}}$-action on ${\El}$. 

Let $\pi_{1,\Delta}\colon \Delta(I)_{\mathcal{X}} \to {\VI}$ be the natural projection. 
By Lemma \ref{lemma: descend to VI} (3), we have 
${\El}|_{\Delta(I)_{\mathcal{X}}} \simeq \pi_{1,\Delta}^{\ast}({\El}|_{{\VI}})$.  
Modular forms restricted to the boundary descend to ${\VI}$ in the following sense. 

\begin{lemma}\label{lem: descend section to VI}
Let $f$ be a modular form of weight $\lambda$, viewed as a section of ${\El}$ over ${\XI}$. 
Then $f$ extends to a holomorphic section of the canonical extension of ${\El}$ over ${\XIS}$. 
Moreover, there exists a ${\GIbar}$-invariant section $f_{I}$ of ${\El}|_{{\VI}}$ such that 
$f|_{\Delta(I)_{\mathcal{X}}}=\pi_{1,\Delta}^{\ast}f_{I}$.  
\end{lemma}

\begin{proof}
The canonical extension over ${\XIS}$ is the pullback of the canonical extension over 
$\mathcal{X}(I_0)^{\Sigma_{I_{0}}}$ by the gluing map 
${\XIS} \to \mathcal{X}(I_0)^{\Sigma_{I_{0}}}$. 
Therefore, as for the first assertion, it suffices to prove the extendability at $I_0$. 
By construction, this is the same as the extendability of $f$ over $\mathcal{X}(I_{0})^{\Sigma_{I_{0}}}$ 
as a ${\Ivl}$-valued function,  
which in turn follows from the Koecher principle $l\in \overline{\mathcal{C}(I^{0})}$ 
in the Fourier expansion \eqref{eqn: Fourier expansion}. 

Since the fibers of $\pi_{1,\Delta}$ are connected union of smooth projective varieties, 
the restriction of $f|_{\Delta(I)_{\mathcal{X}}}$ to each $\pi_{1,\Delta}$-fiber 
as a section of $\pi_{1,\Delta}^{\ast}({\El}|_{{\VI}})$ is constant. 
Therefore $f|_{\Delta(I)_{\mathcal{X}}}$ is the pullback of a section $f_{I}$ of ${\El}|_{{\VI}}$. 
The ${\GIbar}$-invariance of $f|_{\Delta(I)_{\mathcal{X}}}$ implies that of $f_{I}$. 
\end{proof}

\subsection{Full toroidal compactification}\label{ssec: full toroidal}

Let $\Sigma=(\Sigma_{I})_{I}$ be a ${\G}$-admissible collection of fans in the sense of \cite{AMRT} \S III.5. 
Here $I$ ranges over all isotropic subspaces of ${\Q}^{2g}$, 
and each $\Sigma_{I}$ is a ${\GIl}$-admissible cone decomposition of ${\CI}^{\ast}$ 
satisfying suitable compatibility conditions. 
The toroidal compactification $X^{\Sigma}$ of $X$ associated to $\Sigma$ is defined as the quotient space 
\begin{equation*}
X^{\Sigma} = \left(  {\D}\sqcup \bigsqcup_{I}{\XI}^{\Sigma_{I}} \right) / \sim,  
\end{equation*}
where $\sim$ is the equivalence relation generated by 
the gluing maps ${\D}\to {\XIS}$, $\mathcal{X}(J)^{\Sigma_{J}}\to {\XIS}$ for $J\subset I$, 
and the ${\G}$-action. 

It is a fundamental theorem of Ash-Mumford-Rapoport-Tai \cite{AMRT} that 
$X^{\Sigma}$ is a compact analytic space containing $X$ as a Zariski open set, 
and the identity of $X$ extends to a morphism $\pi\colon X^{\Sigma}\to X^{\ast}$ to the Satake compactification. 
We may choose $\Sigma$ so that $X^{\Sigma}$ is smooth projective (\cite{AMRT}), 
where the smoothness is assured by requiring that each $\Sigma_{I}$ is regular. 
Furthermore, we may assume that the boundary divisor $X^{\Sigma}-X$ is simple normal crossing (see \cite{Ma2} Appendix). 

The natural projection ${\XIS}\to X^{\Sigma}$ factors through the quotient by ${\GIbar}$:   
\begin{equation*}
{\XIS}\to {\XIS}/{\GIbar}\to X^{\Sigma}. 
\end{equation*}
The second map is an isomorphism around the $I$-locus $\pi^{-1}(X_{I})$ (see \cite{AMRT} p.175). 
A boundary stratum $\Delta(I, \sigma)_{\mathcal{X}}$ of ${\XIS}$ is mapped into $\pi^{-1}(X_{I})$ if and only if $\sigma$ is an $I$-cone 
(\cite{AMRT} p.164). 
Therefore 
\begin{equation*}
\pi^{-1}(X_{I}) \simeq \Delta(I)_{\mathcal{X}}/{\GIbar}. 
\end{equation*}
For an $I$-cone $\sigma$ 
we denote by $\Delta([\sigma])\subset \pi^{-1}(X_{I})$ the image of $\Delta(I, \sigma)_{\mathcal{X}}$. 
Then $\pi^{-1}(X_{I})$ has the stratification 
\begin{equation*}
\pi^{-1}(X_{I})  =  \bigsqcup_{[\sigma]\in \Sigma_{I}^{\circ}/{\GIl}} \Delta([\sigma]), 
\end{equation*}
and $\Delta([\sigma])$ is isomorphic to $\Delta(I, \sigma)_{\mathcal{Y}}$ by the description in \S \ref{ssec: partial toroidal}. 
Note that the isomorphism $\Delta([\sigma]) \simeq \Delta(I, \sigma)_{\mathcal{Y}}$ 
depends on the cone $\sigma$, rather than its ${\GIl}$-equivalence class $[\sigma]$. 
If we use another equivalent cone $\gamma \sigma$, $\gamma\in {\GIl}$, 
the difference 
\begin{equation*}
\Delta(I, \sigma)_{\mathcal{Y}} \simeq \Delta([\sigma]) \simeq \Delta(I, \gamma \sigma)_{\mathcal{Y}} 
\end{equation*}
is the $\gamma$-action on $\Delta(I, \sigma)_{\mathcal{Y}}\subset {\YIS}$. 
Thus we have a projection 
\begin{equation*}
\Delta([\sigma]) \simeq \Delta(I, \sigma)_{\mathcal{Y}} \to Y_{I} 
\end{equation*}
for each choice of $\sigma$, 
but we cannot glue them over $\Sigma_{I}^{\circ}/{\GIl}$ unless $\dim I=1$ (where ${\GIl}$ is trivial). 
In other words, projection from the whole $\pi^{-1}(X_{I})$ to $Y_{I}$ is not well-defined. 

For a minimal $I$-cone $\sigma$, we denote by 
\begin{equation}\label{eqn: irr comp D(I)}
\overline{\Delta([\sigma])}^{I} 
\simeq \overline{\Delta(I, \sigma)}_{\mathcal{X}}/{\GIbard} 
\simeq \overline{\Delta(I, \sigma)}_{\mathcal{Y}}
\end{equation}
the image of $\overline{\Delta(I, \sigma)}_{\mathcal{X}}$ in $\pi^{-1}(X_{I})$. 
This is the closure of $\Delta([\sigma])$ in $\pi^{-1}(X_{I})$, 
and has the structure of a smooth projective toric fibration over $Y_{I}$ 
(depending on the choice of a cone $\sigma$ representing $[\sigma]$).  
The irreducible decomposition of $\pi^{-1}(X_{I})$ is given by 
\begin{equation*}
\pi^{-1}(X_{I}) = 
\sum_{\substack{[ \sigma ] \in \Sigma_{I}^{\circ}/{\GIl} \\ \textrm{minimal}}} \overline{\Delta([\sigma])}^{I}. 
\end{equation*}

\section{Siegel operator via toroidal compactification}\label{sec: Siegel}

The basic theory of Siegel operators for vector-valued Siegel modular forms was developed by Weissauer in \cite{We1} \S 2. 
In this section we reformulate the Siegel operators at the level of toroidal compactifications. 
In \S \ref{ssec: canonical extension}, we descended modular forms at the boundary to ${\VI}$. 
What we do here is to descend it further to ${\DI}$. 
See \cite{Ma1} Chapter 6 for a similar treatment in the orthogonal case. 

Throughout this section ${\G}$ is a neat arithmetic subgroup of ${\Sp}$. 
We fix an isotropic subspace $I$ of ${\Q}^{2g}$ of dimension $r$ 
and a smooth toroidal compactification $X^{\Sigma}$ of $X={\D}/{\G}$. 
Then we abbreviate ${\XIcpt}={\XIS}$. 
We also choose a maximal isotropic subspace $I_0\subset {\Q}^{2g}$ containing $I$. 
This plays only auxiliary roles, but is also fixed throughout this section.

\subsection{Boundary Fourier coefficients}\label{ssec: Fourier coefficient}

In this subsection we recall some results of Weissauer (\cite{We1} \S 2) in a form adapted to our purpose. 
We denote by $P(I, I_0)$ the stabilizer of $I$ in ${\rm GL}(I_0)$, 
and $U(I, I_0)$ the unipotent radical of $P(I, I_0)$. 
They fit into the exact sequence 
\begin{equation}\label{eqn: PI0I}
0 \to U(I, I_0) \to P(I, I_0) \to {\rm GL}(I)\times {\rm GL}(I_0/I) \to 1, 
\end{equation}
and $U(I, I_0)$ is canonically isomorphic to $I\otimes (I_0/I)^{\vee}$. 
The group $W(I)$ stabilizes $I_{0}$, and 
$U(I, I_{0})$ is the image of the natural map $W(I)\to {\rm GL}(I_{0})$. 

We identify ${\rm GL}(I_0)= {\rm GL}(I^{0})$ naturally where $I^0=I_{0}^{\vee}$ 
and regard $P(I, I_0)$ as a subgroup of ${\rm GL}(I^{0})$. 
Let $I^{\perp}=I^{\perp}\cap I^{0}$ be the annihilator of $I$ in $I^{0}$. 
This is canonically identified with $(I_0/I)^{\vee}$. 
Then $P(I, I_0)$ is the stabilizer of the subspace $I^{\perp}$ of $I^{0}$ in ${\rm GL}(I^{0})$, 
and \eqref{eqn: PI0I} can be rewritten as 
\begin{equation}\label{eqn: PI0I dual}
0 \to U(I, I_0) \to P(I, I_0) \to {\rm GL}(I^{\perp})\times {\rm GL}(I^{\vee}) \to 1 
\end{equation}
where $I^{\vee}$ is identified with the quotient $I^{0}/I^{\perp}$ of $I^{0}$. 

Let $\lambda=(\lambda_{1} \geq \cdots \geq \lambda_{g} \geq 0)$ be a highest weight for 
${\rm GL}(g, {\C})\simeq {\rm GL}(I^{0}_{{\C}})$. 
The number of indices $i$ with $\lambda_{i}=\lambda_{g}$ is called the \textit{corank} of $\lambda$. 
We denote by $({\Ivl})^{U}$ the $U(I, I_{0})_{{\C}}$-invariant part of ${\Ivl}$. 
By \eqref{eqn: PI0I dual}, this is a representation of ${\rm GL}(I^{\perp}_{{\C}})\times {\rm GL}(I_{{\C}}^{\vee})$. 
More specifically, by \cite{We1} Lemma 2, we have 
\begin{equation}\label{eqn: U-inv rep}
({\Ivl})^{U} \simeq I^{\perp}_{{\C},\lambda'} \boxtimes I_{{\C},\lambda''}^{\vee} 
\end{equation}
where $\lambda'=(\lambda_{1}, \cdots, \lambda_{g-r})$ and $\lambda''=(\lambda_{g-r+1}, \cdots, \lambda_{g})$. 
In particular, when $\lambda$ has corank $\geq r$, 
then $I_{{\C},\lambda''}^{\vee}$ is $1$-dimensional and so 
\begin{equation}\label{eqn: U-inv corank >= r}
({\Ivl})^{U} \simeq I^{\perp}_{{\C},\lambda'} 
\end{equation}
as a representation of ${\rm GL}(I^{\perp}_{{\C}})$. 

Let $f=\sum_{l}a(l)q^{l}$ be the Fourier expansion of a modular form of weight $\lambda$ at the $I_{0}$-cusp, 
where $a(l)\in {\Ivl}$ and $l\in \overline{\mathcal{C}(I^{0})}\cap U(I_{0})_{{\Z}}^{\vee}$. 
Via the inclusion $S^{2}(I^{\perp})\subset S^{2}I^{0}$, 
the cone $\overline{\mathcal{C}(I^{0})}$ contains $\overline{\mathcal{C}(I^{\perp})}$ in its boundary. 
This consists of positive-semidefinite forms on $I_{0,{\R}}$ 
whose nullspace contains $I_{{\R}}$ ($\Leftrightarrow$ which vanishes on $I_{{\R}}$). 
Thus we have 
\begin{equation*}
\overline{\mathcal{C}(I^{\perp})} = 
S^2(I^{\perp}_{{\R}}) \cap \overline{\mathcal{C}(I^{0})} = 
(S^2I_{{\R}})^{\perp}\cap \overline{\mathcal{C}(I^{0})}. 
\end{equation*}  

\begin{proposition}[\cite{We1} \S 2]\label{prop: Fourier U-invariant}
Let $l\in \overline{\mathcal{C}(I^{\perp})}$. 
If $\lambda$ has corank $<r$, then $a(l)= 0$. 
If $\lambda$ has corank $\geq r$, then 
$a(l) \in ({\Ivl})^{U}\simeq I^{\perp}_{{\C},\lambda'}$. 
\end{proposition}

\subsection{Automorphic vector bundles at boundary}\label{ssec: cusp reduction}

From now on, we assume that $\lambda$ has corank $\geq r$. 
We define a sub vector bundle ${\ElI}$ of ${\El}$ as 
the image of $({\Ivl})^{U}\otimes {\OD}$ by the $I_{0}$-trivialization 
${\Ivl}\otimes {\OD} \to {\El}$. 
We call this isomorphism $({\Ivl})^{U}\otimes {\OD}\to {\ElI}$ the $I_{0}$-trivialization of ${\ElI}$. 

\begin{lemma}\label{lem: ElI}
The sub vector bundle ${\ElI}$ of ${\El}$ does not depend on the choice of $I_{0}$ and is $P(I)_{{\R}}$-invariant. 
Moreover, ${\ElI}$ descends to a sub vector bundle ${\ElI}|_{{\VI}}$ of ${\El}|_{{\VI}}$, 
and the $I_{0}$-trivialization of ${\ElI}$ descends to an isomorphism 
$({\Ivl})^{U}\otimes \mathcal{O}_{{\VI}}\to {\ElI}|_{{\VI}}$. 
\end{lemma}

\begin{proof}
As for the independence from $I_0$, 
it suffices to show that if we use another $I_{0}'\supset I$, 
the difference of the $I_0$-trivialization and the $I_{0}'$-trivialization 
as linear maps $(I_{0,{\C}})_{\lambda}^{\vee}\to (I_{0,{\C}}')_{\lambda}^{\vee}$ 
sends the $U(I, I_{0})_{{\C}}$-invariant part of $(I_{0,{\C}})_{\lambda}^{\vee}$ 
to the $U(I, I_{0}')_{{\C}}$-invariant part of $(I_{0,{\C}}')_{\lambda}^{\vee}$. 
This amounts to checking that the isomorphism 
${\rm GL}(I^{\vee}_{0,{\C}})\to {\rm GL}((I')^{\vee}_{0,{\C}})$ 
induced by the similar difference 
$I^{\vee}_{0,{\C}} \to (I')^{\vee}_{0,{\C}}$ 
for ${\E}$ sends $U(I, I_{0})_{{\C}}$ to $U(I, I_{0}')_{{\C}}$. 
This in turn holds because  
$I^{\vee}_{0,{\C}} \to (I')^{\vee}_{0,{\C}}$ 
sends 
$I^{\perp}\cap I^{\vee}_{0,{\C}}$ to $I^{\perp}\cap (I')^{\vee}_{0,{\C}}$ 
via $I^{\perp} \cap {\E}$. 
This proves the independence from $I_{0}$. 
The $P(I)_{{\R}}$-invariance can be verified similarly. 

Since $U(I)_{{\C}}$ acts on $I_{0,{\C}}$ trivially, 
${\ElI}$ is a $U(I)_{{\C}}$-invariant sub vector bundle of ${\El}$ and 
the $I_{0}$-trivialization ${\ElI}\to ({\Ivl})^{U} \otimes \mathcal{O}_{{\D}}$ is $U(I)_{{\C}}$-invariant. 
This shows that ${\ElI}$ descends to a sub vector bundle ${\ElI}|_{{\VI}}$ of ${\El}|_{{\VI}}$ 
via the equivariant action of $U(I)_{{\C}}$, 
and the $I_{0}$-trivialization of ${\ElI}$ descends to a trivialization of ${\ElI}|_{{\VI}}$ 
(which is just a restriction of the $I_{0}$-trivialization of ${\El}|_{{\VI}}$). 
\end{proof}


\begin{example}
When ${\El}={\E}$, the fiber of ${\E}^{I}\subset {\E}$ over $[V]\in {\D}$ is $I^{\perp}\cap V \subset V$. 
The $I_{0}$-trivialization sends it to $I^{\perp}\cap I^{0}_{{\C}} \subset I^{0}_{{\C}}$. 
Moreover, ${\E}^{I}$ is isomorphic to the pullback of the Hodge bundle on ${\DI}$. 
Indeed, if $[W]\in {\DI}$ is the image of $[V]$, 
the projection $I^{\perp}\to V(I)$ gives a natural isomorphism $I^{\perp}\cap V \to W$, 
the target being the fiber of the Hodge bundle on ${\DI}$ over $[W]$. 
%
\end{example}

Let $f$ be a modular form of weight $\lambda$. 
Recall from Lemma \ref{lem: descend section to VI} that the restriction of $f$ to 
the $I$-locus $\Delta(I)_{\mathcal{X}}$ in the boundary of $\overline{\mathcal{X}(I)}$
descends to a section $f_{I}$ of ${\El}|_{{\VI}}$. 
The relevance of ${\ElI}$ to the Siegel operator comes from the following refinement of Lemma \ref{lem: descend section to VI}. 

\begin{lemma}\label{lem: MF at boundary}
$f_{I}$ takes values in ${\ElI}|_{{\VI}}$. 
\end{lemma}

\begin{proof}
Let $\sigma\in \Sigma_{I}^{\circ}$ be an arbitrary $I$-cone. 
What has to be shown is that the restriction of $f$ to the boundary stratum $\Delta(I, \sigma)_{\mathcal{X}}$ 
takes values in $({\Ivl})^{U}$ via the $I_{0}$-trivialization. 
By the gluing ${\XIcpt}\to \overline{\mathcal{X}(I_0)}$, 
this is reduced to showing that the restriction of $f$ to the $\sigma$-stratum 
$\Delta(I_{0}, \sigma)_{\mathcal{X}}$ in $\overline{\mathcal{X}(I_0)}$ 
takes values in $({\Ivl})^{U}$ via the $I_{0}$-trivialization. 
Let $f=\sum_{l}a(l)q^{l}$ be the Fourier expansion of $f$ at $I_{0}$. 
Since $\overline{\mathcal{C}(I^0)}$ is the dual cone of $\overline{\mathcal{C}(I_0)}$, 
we have $(l ,\sigma)\geq 0$ for every index vector $l$ with $a(l)\ne 0$. 
When $(l ,\sigma)\not \equiv 0$, $l$ has positive pairing with the interior of $\sigma$, 
so the function $q^{l}$ vanishes at $\Delta(I_{0}, \sigma)_{\mathcal{X}}$. 
Therefore, if we restrict $f$ to $\Delta(I_0, \sigma)_{\mathcal{X}}$, 
only those functions $q^{l}$ with $(l ,\sigma)\equiv 0$ survive. 
Since $\sigma$ is an $I$-cone, we have 
\begin{equation*}
\sigma^{\perp}\cap \overline{\mathcal{C}(I^{0})} = 
\mathcal{C}(I)^{\perp}\cap \overline{\mathcal{C}(I^{0})} = 
(S^2I_{{\R}})^{\perp}\cap \overline{\mathcal{C}(I^{0})} = 
\overline{\mathcal{C}(I^{\perp})}. 
\end{equation*}
Therefore 
\begin{equation}\label{eqn: Fourier f at boundary stratum}
f|_{\Delta(I_0, \sigma)_{\mathcal{X}}} = 
\sum_{l\in \overline{\mathcal{C}(I^{\perp})}} a(l) q^{l}|_{\Delta(I_{0}, \sigma)_{\mathcal{X}}}. 
\end{equation}
By Proposition \ref{prop: Fourier U-invariant}, 
we have $a(l)\in ({\Ivl})^{U}$ for every index vector $l$ appearing here. 
This proves our assertion. 
\end{proof}

Although ${\El}|_{{\VI}}$ no longer descends to ${\DI}$, 
we show that the sub vector bundle ${\ElI}|_{{\VI}}$ does so. 
Let ${\E}_{\lambda'}$ be the automorphic vector bundle on ${\DI}$ of weight 
$\lambda'=(\lambda_{1}, \cdots, \lambda_{g-r})$. 

\begin{proposition}\label{prop: ElI descend}
${\ElI}|_{{\VI}}\simeq \pi_{2}^{\ast}{\E}_{\lambda'}$ 
as $P(I)_{{\R}}$-equivariant vector bundles. 
\end{proposition}

\begin{proof}
Let $j(\gamma, x)$ be the factor of automorphy for the $P(I)_{{\R}}$-action on ${\ElI}|_{{\VI}}$ 
with respect to the $I_{0}$-trivialization of ${\ElI}|_{{\VI}}$. 
By construction, this is the descent of the one for ${\ElI}$. 
Since $W(I)_{{\R}}$ stabilizes $I_{0,{\R}}$ and the image of $W(I)_{{\R}}\to {\rm GL}(I_{0,{\R}})$ is $U(I, I_{0})_{{\R}}$, 
we have $j(\gamma, \cdot) \equiv {\rm id}$ for every $\gamma \in W(I)_{{\R}}$ 
by the definition of ${\ElI}$. 
Since $W(I)_{{\R}}$ acts on each $\pi_{2}$-fiber transitively,  
the same argument as the proof of Lemma \ref{lem: f.a. constant} shows that 
$j(\gamma, x) = j(\gamma, x')$ if $\pi_{2}(x)=\pi_{2}(x')$ for any $\gamma \in P(I)_{{\R}}$. 
Then we can argue as in the proof of Lemma \ref{lemma: descend to VI}: 
the function $j(\gamma, x)$ descends to a function on $P(I)_{{\R}}\times {\DI}$ 
which defines a $P(I)_{{\R}}$-equivariant vector bundle on ${\DI}$. 
($W(I)_{{\R}}$ acts trivially.) 
This is the descent of ${\ElI}|_{{\VI}}$ to ${\DI}$. 
We denote it by $\Phi_{I}{\El}$. 
In this descent, the fibers of ${\ElI}|_{{\VI}}$ over points in the same $\pi_{2}$-fiber are identified 
by the equivariant $W(I)_{{\R}}$-action. 

By construction $W(I)_{{\R}}$ acts on $\Phi_{I}{\El}$ trivially, 
and \eqref{eqn: U-inv rep} shows that 
${\rm GL}(I_{{\R}})$ acts on $\Phi_{I}{\El}$ as scalar multiplication. 
We calculate the weight of $\Phi_{I}{\El}$ with respect to ${\rm Sp}(V(I))$. 
We choose a base point $[V]\in {\D}$ and let $[W]\in {\DI}$ be the image of $[V]$. 
Recall from \S \ref{ssec: Siegel domain} that $W$ is the injective image of $I^{\perp}\cap V$ in $V(I)_{{\C}}$. 
Since the $I_{0}$-trivialization $V\to I^{0}_{{\C}}$ sends 
$I^{\perp}\cap V$ to $I^{\perp}\cap I^{0}_{{\C}}$, 
the $I_{0}$-trivialization of $V_{\lambda}$ sends 
$({\Ivl})^{U}$ to the invariant part of $V_{\lambda}$ for the unipotent radical of the stabilizer of 
$I^{\perp}\cap V$ in ${\rm GL}(V)$. 
This shows that the latter is the fiber of ${\ElI}$ over $[V]$ and hence the fiber of $\Phi_{I}{\El}$ over $[W]$. 
By \eqref{eqn: U-inv corank >= r}, this is isomorphic to 
$(I^{\perp}\cap V)_{\lambda'} \simeq W_{\lambda'}$ as a representation of ${\rm GL}(W)$. 
This implies $\Phi_{I}{\El}\simeq {\E}_{\lambda'}$. 
\end{proof}

By construction $\Phi_{I}{\El}$ comes with a $I_{0}$-trivialization 
$\Phi_{I}{\El} \simeq ({\Ivl})^{U}\otimes \mathcal{O}_{{\DI}}$, 
which is the descent of the $I_{0}$-trivialization of ${\ElI}|_{{\VI}}$.

\subsection{Siegel operator}\label{ssec: Siegel operator}

We can now formulate the Siegel operator at the level of ${\XIcpt}$. 
Let $f$ be a modular form of weight $\lambda$, 
which we view as a section of ${\El}$ over ${\XIcpt}$ by Lemma \ref{lem: descend section to VI}. 
Let $\Delta(I)_{\mathcal{X}}$ be the $I$-locus in the boundary of ${\XIcpt}$ as defined in \eqref{eqn: D(I)X}. 
We denote by $\pi_{\Delta}\colon \Delta(I)_{\mathcal{X}} \to {\DI}$ the projection. 
By Lemma \ref{lem: ElI} and Proposition \ref{prop: ElI descend}, 
we have a ${\GIbar}$-equivariant isomorphism 
\begin{equation}\label{eqn: ElI boundary}
{\ElI}|_{\Delta(I)_{\mathcal{X}}}\simeq \pi_{\Delta}^{\ast}\Phi_{I}{\El} \simeq \pi_{\Delta}^{\ast}{\E}_{\lambda'}. 
\end{equation}

\begin{proposition}\label{prop: Siegel operator}
There exists a modular form $\Phi_{I}(f)$ of weight $\lambda'$ on ${\DI}$ with respect to ${\GIh}$ such that 
$f|_{\Delta(I)_{\mathcal{X}}}=\pi_{\Delta}^{\ast}\Phi_{I}(f)$ as sections of \eqref{eqn: ElI boundary}. 
If $f=\sum_{l}a(l)q^{l}$ is the Fourier expansion of $f$ at the $I_{0}$-cusp, 
the Fourier expansion of $\Phi_{I}(f)$ at the $I_{0}/I$-cusp of ${\DI}$ is given by 
\begin{equation}\label{eqn: Siegel Fourier}
\Phi_{I}(f) = \sum_{l\in \overline{\mathcal{C}(I^{\perp})}} a(l)q^{l}. 
\end{equation}
\end{proposition}

Here we view $q^l$ for $l\in \overline{\mathcal{C}(I^{\perp})}=\overline{\mathcal{C}((I_{0}/I)^{\vee})}$ 
as a function on ${\DI}$ via the tube domain realization ${\DI}\subset S^2(I_{0}/I)_{{\C}}$ of ${\DI}$ 
with respect to $I_{0}/I \subset V(I)$. 

\begin{proof}
By Lemma \ref{lem: MF at boundary}, the descent $f_{I}$ of $f|_{\Delta(I)_{\mathcal{X}}}$ to ${\VI}$ takes values in 
${\ElI}|_{{\VI}}\simeq \pi_{2}^{\ast}{\E}_{\lambda'}$. 
Since $f_{I}$ is ${\GIbar}$-invariant, 
it is in particular invariant under the translation by $\overline{W(I)}_{{\Z}}$ 
on the fibers of $\pi_{2}\colon {\VI}\to {\DI}$ 
and so descends to ${\VI}/\overline{W(I)}_{{\Z}}$. 
Since ${\VI}/\overline{W(I)}_{{\Z}} \to {\DI}$ has compact fibers (abelian varieties), 
we find that $f_{I}=\pi_{2}^{\ast}\Phi_{I}(f)$ 
for a section $\Phi_{I}(f)$ of ${\E}_{\lambda'}$ over ${\DI}$.  
Hence $f|_{\Delta(I)_{\mathcal{X}}}=\pi_{\Delta}^{\ast}\Phi_{I}(f)$. 
The ${\GIbar}$-invariance of $f_{I}$ implies the ${\GIh}$-invariance of $\Phi_{I}(f)$. 

It remains to verify \eqref{eqn: Siegel Fourier}. 
By \eqref{eqn: Fourier f at boundary stratum}, the $I_{0}$-trivialization of $\pi_{\Delta}^{\ast}\Phi_{I}(f)$ is expressed as  
\begin{equation*}
\pi_{\Delta}^{\ast}\Phi_{I}(f) = f|_{\Delta(I)_{\mathcal{X}}} 
= \sum_{l\in \overline{\mathcal{C}(I^{\perp})}} a(l)q^{l}|_{\Delta(I)_{\mathcal{X}}}, 
\quad a(l)\in ({\Ivl})^{U}. 
\end{equation*}
Here, for $l\in \overline{\mathcal{C}(I^{\perp})}\cap U(I_{0})_{{\Z}}^{\vee}$, 
we view the function $q^{l}$ on $\overline{\mathcal{X}(I_{0})}$ as a function on ${\XIcpt}$ 
by the pullback by the gluing map ${\XIcpt}\to \overline{\mathcal{X}(I_{0})}$. 
This descends to a function on ${\DI}\subset S^2(I_0/I)_{{\C}}$ 
by the condition $l\in \overline{\mathcal{C}(I^{\perp})}\subset S^2(I_{0}/I)_{{\R}}^{\vee}$. 
By the construction of $\Phi_{I}{\El}$, 
pullback of sections by $\pi_{\Delta}$ is identified with ordinary pullback of $({\Ivl})^{U}$-valued functions by the $I_0$-trivialization. 
Therefore the $I_0$-trivialization of $\Phi_{I}(f)$ has the Fourier expansion \eqref{eqn: Siegel Fourier}. 
\end{proof}

We call the map 
\begin{equation*}
\Phi_{I} : M_{\lambda}({\G}) \to M_{\lambda'}({\GIh}), \qquad 
f\mapsto \Phi_{I}(f), 
\end{equation*}
the \textit{Siegel operator} at the $I$-cusp. 
This can be understood as a two-step operation: 
restriction to the toroidal boundary $\Delta(I)_{\mathcal{X}}$, and 
then descent to ${\DI}$. 
By the calculation \eqref{eqn: Siegel Fourier} of Fourier expansion, 
this agrees with the classical definition of Siegel operator after taking the $I_{0}$-trivialization. 

\section{Holomorphic differential forms}\label{sec: Omega}

In this section we prove Theorem \ref{thm: main} 
by applying the result of \S \ref{sec: Siegel} to holomorphic differential forms 
and combining it with consideration of filtrations associated to the Siegel domain realization. 
The essential points are contained in the representation-theoretic calculation of the relevant vector bundles in 
\S \ref{ssec: first filtration} and \S \ref{ssec: holomorphic Leray}. 

Throughout this section, $X={\D}/{\G}$ is a Siegel modular variety of genus $g>1$ with ${\G}$ neat. 
We fix an isotropic subspace $I$ of ${\Q}^{2g}$ of dimension $r$, 
together with an auxiliary maximal isotropic subspace $I_{0}\supset I$. 

\subsection{Weissauer weights}\label{ssec: Weissauer weight}

Let $0\leq \alpha \leq g$. 
Let ${\la}$ be the highest weight for ${\rm GL}(g, {\C})$ defined by 
\begin{equation*}
{\la} = ((g+1)^{g-\alpha}, (g-\alpha)^{\alpha}), 
\end{equation*}
where the notation means that $g+1$ repeats $g-\alpha$ times and 
then $g-\alpha$ repeats $\alpha$ times. 
When $\alpha >0$, ${\la}$ has corank $\alpha$. 
We write 
\begin{equation*}
k(\alpha) = |{\la}|/2 = g(g+1)/2-\alpha(\alpha+1)/2. 
\end{equation*}
The ${\rm GL}(g, {\C})$-representation 
$\wedge^{k(\alpha)}S^2{\C}^{g}$ 
contains an irreducible representation of highest weight ${\la}$ with multiplicity $1$ (see \cite{We1} \S 7). 
We denote it by ${\C}^{g}_{{\la}}$. 
Correspondingly, 
$\wedge^{k(\alpha)}S^2{\E}\simeq \Omega^{k(\alpha)}_{{\D}}$ 
contains the automorphic vector bundle ${\Ela}$ with multiplicity $1$. 

In \cite{We1} \S 7, Weissauer proved that $H^{0}(\Omega_{X}^{k})=0$ 
unless $k=k(\alpha)$ for some $0\leq \alpha \leq g$. 
Moreover, he proved that global sections of $\Omega^{k(\alpha)}_{X}$ 
take values in the sub vector bundle ${\Ela}$, and hence 
\begin{equation*}
H^0(\Omega_{X}^{k(\alpha)}) = M_{{\la}}({\G}). 
\end{equation*}

For later use, let us calculate a highest weight vector of ${\C}^{g}_{{\la}}$. 

\begin{lemma}\label{lem: highest weight vector}
Let $e_{1}, \cdots, e_{g}$ be the standard basis of ${\C}^{g}$. 
Then the vector 
\begin{equation*}
e_{{\la}} := \bigwedge_{\substack{i\leq j \\ i\leq g-\alpha}} e_{i}e_{j} \quad 
\in \wedge^{k(\alpha)}S^2{\C}^{g} 
\end{equation*}
is a highest weight vector of ${\C}^{g}_{{\la}}$. 
\end{lemma}

\begin{proof}
In the definition of $e_{{\la}}$, each vector $e_i$ with $1\leq i \leq g-\alpha$ appears $g+1$ times, 
while $e_{i}$ with $g-\alpha < i \leq g$ appears $g-\alpha$ times. 
This shows that $e_{{\la}}$ is an eigenvector of weight ${\la}$ for the diagonal torus $({\C}^{\ast})^{g}\subset {\rm GL}(g, {\C})$. 
Moreover, direct calculation shows that $e_{{\la}}$ is fixed by unipotent upper triangular matrices. 
Thus $e_{{\la}}$ is a highest weight vector of weight ${\la}$. 
The irreducible representation generated by $e_{{\la}}$ 
must coincide with ${\C}^{g}_{{\la}}$ by the multiplicity $1$ property. 
\end{proof}

\subsection{First filtration}\label{ssec: first filtration}

From now on, we fix $\alpha$ with $r\leq \alpha < g$ and write $k=k(\alpha)$. 
The Siegel domain realization 
${\D}\stackrel{\pi_{1}}{\to} {\VI}\stackrel{\pi_{2}}{\to} {\DI}$ 
with respect to $I$ defines the filtration 
\begin{equation}\label{eqn: Siegel domain filtration}
\pi_{1}^{\ast}\pi_{2}^{\ast}\Omega_{{\DI}}^k \subset \pi_{1}^{\ast}\Omega_{{\VI}}^k \subset \Omega_{{\D}}^k 
\end{equation}
on $\Omega_{{\D}}^k$. 
As in \S \ref{sec: Siegel}, we write $I^{0}=I_{0}^{\vee}$ and $I^{\perp}=(I_{0}/I)^{\vee}\subset I^{0}$.

\begin{lemma}
By the $I_{0}$-trivialization of $\Omega_{{\D}}^{k}$, 
the filtration \eqref{eqn: Siegel domain filtration} corresponds to the constant filtration 
\begin{equation}\label{eqn: Siegel domain filtration constant}
\wedge^kS^2(I^{\perp}) \subset \wedge^k(I^{\perp}\cdot I^0) \subset \wedge^kS^2I^0 
\end{equation}
on $\wedge^{k}S^2I^{0}$ tensored with ${\C}$. 
\end{lemma}

\begin{proof}
Recall from \S \ref{ssec: Siegel domain} that 
the tube domain realization at $I_{0}$ identifies the Siegel domain realization with 
a restriction of the linear projections 
\begin{equation*}
S^2I_{0, {\C}} \to S^2I_{0, {\C}}/S^2I_{{\C}} \to S^2(I_0/I)_{{\C}}.  
\end{equation*}
In view of Example \ref{ex: I-trivialization TD}, 
these linear maps correspond to the bundle maps 
$T_{{\D}}\to \pi_{1}^{\ast}T_{{\VI}} \to \pi_{1}^{\ast}\pi_{2}^{\ast}T_{{\DI}}$ 
by the $I_{0}$-trivialization of $T_{{\D}}$. 
Hence the dual filtration 
\begin{equation*}
S^2(I^{\perp}_{{\C}}) \subset I_{{\C}}^{\perp}\cdot I^{0}_{{\C}} \subset S^2I^{0}_{{\C}} 
\end{equation*}
corresponds to the filtration 
$\pi_{1}^{\ast}\pi_{2}^{\ast}\Omega_{{\DI}}^1 \subset \pi_{1}^{\ast}\Omega_{{\VI}}^1 \subset \Omega_{{\D}}^1$ 
on $\Omega_{{\D}}^1$ by the $I_0$-trivialization of $\Omega_{{\D}}^1$. 
\end{proof}



We relate the result of \S \ref{sec: Siegel} for ${\Ela}$ to the filtration \eqref{eqn: Siegel domain filtration}. 
The sub vector bundle ${\Ela}$ of $\Omega_{{\D}}^{k}$ is not contained in $\pi_{1}^{\ast}\Omega_{{\VI}}^{k}$. 
However, $\mathcal{E}_{{\la}}^{I}$ is so: 

\begin{proposition}\label{prop: ElaI initial filter}
We have ${\E}_{{\la}}^{I}\subset \pi_{1}^{\ast}\Omega_{{\VI}}^{k}$. 
\end{proposition}

\begin{proof}
The $I_0$-trivialization translates this statement to the assertion in linear algebra that 
the $U(I, I_0)_{{\C}}$-invariant part of the ${\la}$-component of $\wedge^{k}S^2I^{0}_{{\C}}$ 
is contained in $\wedge^k(I^{\perp}\cdot I^{0})_{{\C}}$. 
We take a basis $e_1, \cdots, e_g$ of $I^{0}$ such that 
$I^{\perp}$ is spanned by $e_1, \cdots, e_{g-r}$. 
We identify $I^{0}={\Q}^{g}$ and $I^{\perp}={\Q}^{g-r}$ accordingly. 
Then what has to be shown is 
\begin{equation}\label{eqn: VlaU in filter}
({\C}^{g}_{{\la}})^{U} \subset \wedge^{k}({\C}^{g-r}\cdot {\C}^{g}), 
\end{equation}
where $U$ is the unipotent radical of the stabilizer $P\subset {\rm GL}(g, {\C})$ of ${\C}^{g-r}\subset {\C}^{g}$. 

We first observe that the highest weight vector $e_{{\la}}$ of 
${\C}^{g}_{{\la}}$ is contained in $\wedge^{k}({\C}^{g-r}\cdot {\C}^{g})$. 
Indeed, we have $i\leq g-\alpha \leq g-r$ for all indices $(i, j)$ appearing in the definition of $e_{{\la}}$. 
This means that all components $e_{i}e_{j}$ of $e_{{\la}}$ are contained in ${\C}^{g-r}\cdot {\C}^{g}$. 
Hence $e_{{\la}} \in \wedge^{k}({\C}^{g-r}\cdot {\C}^{g})$. 

Since the subspace $\wedge^{k}({\C}^{g-r}\cdot {\C}^{g})$ of $\wedge^{k}S^2{\C}^{g}$ 
is preserved by $P$, we see that 
\begin{equation}\label{eqn: Pela}
P \cdot e_{{\la}} \subset \wedge^{k}({\C}^{g-r}\cdot {\C}^{g}). 
\end{equation}
The parabolic subgroup $P$ has the structure 
\begin{equation*}
P = ({\rm GL}(g-r, {\C})\times {\rm GL}(r, {\C}))\ltimes U. 
\end{equation*}
The unipotent radical of the standard Borel subgroup of ${\rm GL}(g, {\C})$ is 
generated by $U$ and those of ${\rm GL}(g-r, {\C})$ and ${\rm GL}(r, {\C})$. 
This shows that $e_{{\la}}$ is contained in $({\C}^{g}_{{\la}})^{U}$ and 
is also a highest weight vector of $({\C}^{g}_{{\la}})^{U}$ as a representation of 
${\rm GL}(g-r, {\C})\times {\rm GL}(r, {\C})$. 
(Actually ${\rm GL}(r, {\C})$ acts as scalar multiplications by \eqref{eqn: U-inv rep}.) 
Therefore $({\C}^{g}_{{\la}})^{U}$ is generated by $e_{{\la}}$ as a representation of $P$. 
Then \eqref{eqn: Pela} implies \eqref{eqn: VlaU in filter}. 
\end{proof}

We take the descent $\Omega_{{\D}}^{k}|_{{\VI}}$ of $\Omega_{{\D}}^{k}$ to ${\VI}$ as in Lemma \ref{lemma: descend to VI}. 
The filtration \eqref{eqn: Siegel domain filtration} descends to 
\begin{equation*}
\pi_{2}^{\ast}\Omega_{{\DI}}^{k} \subset \Omega_{{\VI}}^{k} \subset \Omega_{{\D}}^{k}|_{{\VI}},  
\end{equation*}
and this corresponds to \eqref{eqn: Siegel domain filtration constant} 
by the $I_0$-trivialization 
$\wedge^kS^2I_{{\C}}^{0}\otimes \mathcal{O}_{{\VI}}\simeq \Omega^{k}_{{\D}}|_{{\VI}}$ 
of $\Omega_{{\D}}^{k}|_{{\VI}}$. 
By Proposition \ref{prop: ElaI initial filter}, 
the descent ${\E}_{{\la}}^{I}|_{{\VI}}$ of ${\E}_{{\la}}^{I}$, a priori a sub vector bundle of $\Omega_{{\D}}^{k}|_{{\VI}}$, 
is in fact contained in $\Omega_{{\VI}}^{k}$.






\subsection{Holomorphic Leray filtration}\label{ssec: holomorphic Leray}

We refine the first part 
$\wedge^kS^2(I^{\perp}) \subset \wedge^k(I^{\perp}\cdot I^0)$ 
of \eqref{eqn: Siegel domain filtration constant} by taking the decreasing filtration 
\begin{equation*}
L^{l} (\wedge^{k}(I^{\perp}\cdot I^{0})) = 
\wedge^{l}S^{2}(I^{\perp}) \wedge \wedge^{k-l}(I^{\perp}\cdot I^{0}), \qquad 0\leq l \leq k. 
\end{equation*}
By the $I_0$-trivialization, this corresponds to the holomorphic Leray filtration 
\begin{equation*}
L^{l}\Omega^{k}_{{\VI}} = \pi_{2}^{\ast}\Omega_{{\DI}}^{l}\wedge \Omega_{{\VI}}^{k-l} 
\end{equation*}
on $\Omega_{{\VI}}^{k}$ with respect to $\pi_{2}\colon {\VI} \to {\DI}$. 
The graded quotients are respectively isomorphic to 
\begin{equation*}
{\rm Gr}^{l}_{L} (\wedge^{k}(I^{\perp}\cdot I^{0})) \simeq 
\wedge^{l}S^2(I^{\perp}) \otimes \wedge^{k-l}(I^{\perp}\otimes I^{\vee}), 
\end{equation*}
\begin{equation*}
{\rm Gr}^{l}_{L}\Omega^{k}_{{\VI}} \simeq \pi_{2}^{\ast}\Omega_{{\DI}}^{l}\otimes \Omega_{\pi_{2}}^{k-l}, 
\end{equation*}
where $I^{\vee}$ is identified with the quotient $I^0/I^{\perp}$ of $I^0$ 
and $\Omega_{\pi_{2}}^{1}$ is the relative cotangent bundle for $\pi_{2}$. 
We shall determine the level where ${\E}_{{\la}}^{I}|_{{\VI}}$ lives. 

\begin{lemma}\label{prop: Leray level}
We have ${\E}_{{\la}}^{I}|_{{\VI}} \subset L^{k-r(g-\alpha)}\Omega^{k}_{{\VI}}$, 
and the projection 
\begin{equation*}
{\E}_{{\la}}^{I}|_{{\VI}} \to {\rm Gr}_{L}^{k-r(g-\alpha)}\Omega^{k}_{{\VI}} 
\end{equation*}
is injective. 
\end{lemma}

\begin{proof}
We reuse the notation in the proof of Proposition \ref{prop: ElaI initial filter}. 
The $I_0$-trivialization translates the problem to calculating the level of $({\C}^{g}_{{\la}})^{U}$ 
in the filtration 
\begin{equation}\label{eqn: holomorphic Leray filtration trivialized}
L^{l}(\wedge^{k}({\C}^{g-r}\cdot {\C}^{g})) = 
\wedge^{l}S^2{\C}^{g-r}  \wedge  \wedge^{k-l}({\C}^{g-r}\cdot {\C}^{g}) 
\end{equation}
on $\wedge^{k}({\C}^{g-r}\cdot {\C}^{g})$. 

Let $T\subset {\rm GL}(r, {\C})$ be the $1$-dimensional torus of scalar matrices. 
Then $T$ acts on 
\begin{equation*}
{\rm Gr}^{l}_{L} (\wedge^{k}({\C}^{g-r}\cdot {\C}^{g})) \simeq 
\wedge^{l}S^2{\C}^{g-r} \otimes \wedge^{k-l}({\C}^{g-r}\boxtimes {\C}^{r}) 
\end{equation*}
by the scalar multiplication of weight $k-l$. 
Hence \eqref{eqn: holomorphic Leray filtration trivialized} is the filtration by $T$-weights. 
On the other hand, by \eqref{eqn: U-inv rep}, we have 
\begin{equation*} 
({\C}^{g}_{{\la}})^{U} \simeq 
{\C}^{g-r}_{{\la}'}\boxtimes {\C}^{r}_{{\la}''} = 
{\C}^{g-r}_{{\la}'}\boxtimes {\rm det}^{g-\alpha}
\end{equation*}
as representations of ${\rm GL}(g-r, {\C})\times {\rm GL}(r, {\C})$. 
Hence $T$ acts on $({\C}^{g}_{{\la}})^{U}$ by the scalar multiplication of weight $r(g-\alpha)$. 
This shows that 
$({\C}^{g}_{{\la}})^{U}$ is contained in $L^{k-r(g-\alpha)}(\wedge^{k}({\C}^{g-r}\cdot {\C}^{g}))$ 
and the projection 
\begin{equation*}
({\C}^{g}_{{\la}})^{U} \to {\rm Gr}_{L}^{k-r(g-\alpha)}(\wedge^{k}({\C}^{g-r}\cdot {\C}^{g})) 
\end{equation*}
is injective. 
\end{proof}

\subsection{Proof of Theorem \ref{thm: main}}\label{ssec: proof Theorem 1.1}

We take a smooth projective toroidal compactification $X^{\Sigma}$ of $X$ with SNC boundary divisor. 
What have been done are calculations about vector bundles on ${\D}$ and ${\VI}$. 
We translate them to the proof of Theorem \ref{thm: main}. 
We need to distinguish restriction as a section of $\Omega^{k}$ 
and restriction as a differential form. 

Let $\omega$ be a holomorphic $k$-form on $X$ where $k=k(\alpha)$ with $r \leq \alpha < g$. 
We use the same notation $\omega$ for its pullback to ${\XI}$. 
By the theorem of Freitag-Pommerening \cite{FP}, 
$\omega$ extends holomorphically over $X^{\Sigma}$ and hence over ${\XIcpt}$. 
We denote by $\pi_{1}\colon {\XIcpt} \to {\VI}$ the projection and  
write $\overline{\Delta(\sigma)} = \overline{\Delta(I, \sigma)}_{\mathcal{X}}$ 
for an irreducible component of $\Delta(I)_{\mathcal{X}}$ 
as defined in \eqref{eqn: irr comp D(I)X} where $\sigma$ is a minimal $I$-cone. 
Recall that $\overline{\Delta(\sigma)}$ is a smooth projective toric fibration over ${\VI}$. 

\begin{lemma}\label{prop: extension pullback}
There exists a holomorphic $k$-from $\omega_{I}$ on ${\VI}$ such that 
\begin{equation}\label{eqn: restriction=pullback}
\omega|_{\overline{\Delta(\sigma)}} = \pi_{1}^{\ast}\omega_{I}|_{\overline{\Delta(\sigma)}}  
\end{equation}
as sections of $\Omega^{k}_{{\XIcpt}}|_{\overline{\Delta(\sigma)}}$, 
where $|_{\overline{\Delta(\sigma)}}$ means 
the restriction as a section of $\Omega^{k}_{{\XIcpt}}$. 
The $k$-form $\omega_{I}$ does not depend on $\sigma$ and is ${\GIbar}$-invariant.   
\end{lemma}

\begin{proof}
Let us abbreviate $\Delta'=\Delta(I)_{\mathcal{X}}'$. 
We have the inclusion  
\begin{equation*}
\pi_{1}^{\ast}\Omega_{{\VI}}^{k} \: \subset \: \Omega_{{\XIcpt}}^{k} \: \subset \: 
\Omega^{k}_{{\XIcpt}}(\log \Delta') =  \pi_{1}^{\ast} (\Omega^{k}_{{\D}}|_{{\VI}}) 
\end{equation*}
of sheaves on ${\XIcpt}$, 
where $\pi_{1}^{\ast}\Omega_{{\VI}}^{k}$ is a sub vector bundle of both 
$\Omega_{{\XIcpt}}^{k}$ and $\Omega^{k}_{{\XIcpt}}(\log \Delta')$. 
By Lemma \ref{lem: descend section to VI}, 
the restriction of $\omega$ to $\overline{\Delta(\sigma)}$ as a section of $\Omega^{k}_{{\XIcpt}}(\log \Delta')$ 
is the pullback of a section $\omega_{I}$ of $\Omega^{k}_{{\D}}|_{{\VI}}$ 
which does not depend on $\sigma$ and is ${\GIbar}$-invariant. 
By Lemma \ref{lem: MF at boundary} and Proposition \ref{prop: ElaI initial filter}, 
$\omega_{I}$ takes values in the sub vector bundle $\Omega_{{\VI}}^{k}$ of $\Omega^{k}_{{\D}}|_{{\VI}}$. 
This proves our assertion. 
\end{proof}

\begin{remark}\label{remark: FP new proof}
This property in fact leads to a new proof of the theorem of Freitag-Pommerening. 
Indeed, when $\dim \sigma=1$, the fact that 
$\omega|_{\Delta(\sigma)}$ 
takes values in $\pi_{1}^{\ast}\Omega_{{\VI}}^{k}|_{\Delta(\sigma)}$ 
implies the extendability over $\Delta(\sigma)$: 
if $\Delta(\sigma)$ is locally defined by $t=0$, 
the coefficient of a component of $\omega$ that contains $dt/t$ 
must vanish at $\Delta(\sigma)$ by this condition. 
This implies that $\omega$, as a $k$-form on $X$, 
extends over the codimension $1$ strata of $X^{\Sigma}-X$ lying over $X_{I}$. 
Doing this argument for all $I$, we find that $\omega$ extends over $X^{\Sigma}$. 
\end{remark}

Next we consider the cotangent bundle of $\overline{\Delta(\sigma)}$ itself. 
We have a homomorphism 
$\Omega^{k}_{{\XIcpt}}|_{\overline{\Delta(\sigma)}} \to \Omega^{k}_{\overline{\Delta(\sigma)}}$ 
giving the restriction as differential forms. 
The composition 
\begin{equation*}
\pi_{1}^{\ast}\Omega_{{\VI}}^{k}|_{\overline{\Delta(\sigma)}} \: \hookrightarrow \: 
\Omega^{k}_{{\XIcpt}}|_{\overline{\Delta(\sigma)}} \: \to \: \Omega^{k}_{\overline{\Delta(\sigma)}} 
\end{equation*}
gives the pullback by $\overline{\Delta(\sigma)} \to {\VI}$ as differential forms. 
Therefore, sending \eqref{eqn: restriction=pullback} by this homomorphism, 
we obtain 

\begin{corollary}\label{cor: restriction=pullback k-form}
The restriction of $\omega$ to $\overline{\Delta(\sigma)}$ 
as a differential form is the pullback of $\omega_{I}$ by $\overline{\Delta(\sigma)} \to {\VI}$. 
\end{corollary} 


We pass to $\pi^{-1}(X_{I})$. 
By the ${\GIbard}$-invariance of $\omega_{I}$, 
it descends to a $k$-form on $Y_{I}$ (again denoted by $\omega_{I}$). 
Let $\overline{\Delta([\sigma])}^{I}$ 
be an irreducible component of $\pi^{-1}(X_{I})$ as defined in \eqref{eqn: irr comp D(I)}. 
By Corollary \ref{cor: restriction=pullback k-form}, 
the restriction of $\omega$ to $\overline{\Delta([\sigma])}^{I}$ as a differential form is the pullback of $\omega_{I}$ by 
$\overline{\Delta([\sigma])}^{I} \to Y_{I}$.  
Thus the $k$-form $\omega_{I}$ on $Y_I$ coincides with the one considered in \S \ref{sec: intro}. 

Let ${\E}_{{\la}}^{I}|_{Y_{I}}$ be the descent of ${\E}_{{\la}}^{I}|_{{\VI}}$ to $Y_{I}$. 
This is a sub vector bundle of $\Omega^{k}_{Y_{I}}$ by Proposition \ref{prop: ElaI initial filter}, 
and isomorphic to $\pi_{I}^{\ast}\mathcal{E}_{{\la}'}$ by Proposition \ref{prop: ElI descend} 
where $\pi_{I}\colon Y_{I}\to X_{I}'$ is the projection. 
By Lemma \ref{lem: MF at boundary}, the $k$-form $\omega_{I}$ takes values in ${\E}_{{\la}}^{I}|_{Y_{I}}$.  
By Proposition \ref{prop: Siegel operator}, we have 
\begin{equation}\label{eqn: fundamental}
\omega_{I} = \pi_{I}^{\ast}\Phi_{I}(\omega) 
\end{equation}
as sections of ${\E}_{{\la}}^{I}|_{Y_{I}}\simeq \pi_{I}^{\ast}{\E}_{{\la}'}$. 

It remains to consider the holomorphic Leray filtration of $Y_{I}\to X_{I}'$. 
We work with $\tilde{Y}_{I}={\VI}/\overline{W(I)}_{{\Z}}$ and use the symbol 
$\pi_{I}$ also for the projection $\tilde{Y}_{I}\to {\DI}$. 
This is a smooth abelian fibration. 
The group $W(I)_{{\R}}/U(I)_{{\R}}$ acts on $\tilde{Y}_{I}$ 
(this is the reason we work with $\tilde{Y}_{I}$ rather than $Y_{I}$), 
which on each $\pi_{I}$-fiber is by translation. 
Then ${\E}_{{\la}}^{I}|_{\tilde{Y}_{I}}$ is a $W(I)_{{\R}}/U(I)_{{\R}}$-invariant sub vector bundle of $\Omega^{k}_{\tilde{Y}_{I}}$. 
%
Let $L^{\bullet}\Omega_{\tilde{Y}_{I}}^{k}$ be the holomorphic Leray filtration 
on $\Omega_{\tilde{Y}_{I}}^{k}$ with respect to $\pi_{I}$. 
By Lemma \ref{prop: Leray level}, 
${\E}_{{\la}}^{I}|_{\tilde{Y}_{I}}$ is contained in $L^{l}\Omega_{\tilde{Y}_{I}}^{k}$ where $l=k-r(g-\alpha)$, 
and the projection ${\E}_{{\la}}^{I}|_{\tilde{Y}_{I}}\to {\rm Gr}_{L}^{l}\Omega_{\tilde{Y}_{I}}^{k}$ is injective. 

\begin{lemma}\label{lemma: GrL}
We have a $W(I)_{{\R}}/U(I)_{{\R}}$-equivariant isomorphism 
\begin{equation*}
{\rm Gr}^{l}_{L}\Omega_{\tilde{Y}_{I}}^{k} \simeq \pi_{I}^{\ast}( \Omega_{{\DI}}^{l} \otimes \pi_{I \ast}\Omega_{\pi_{I}}^{k-l}), 
\end{equation*} 
and the embedding 
${\E}_{{\la}}^{I}|_{\tilde{Y}_{I}}\hookrightarrow {\rm Gr}_{L}^{l}\Omega_{\tilde{Y}_{I}}^{k}$ 
descends to an embedding 
\begin{equation*}
{\E}_{{\la}'}\hookrightarrow \Omega_{{\DI}}^{l} \otimes \pi_{I \ast}\Omega_{\pi_{I}}^{k-l} 
\end{equation*}
over ${\DI}$. 
\end{lemma}

\begin{proof}
Since $\pi_{I}$ is a smooth abelian fibration, we have 
$\Omega_{A}^{k-l}=H^0(\Omega_{A}^{k-l})\otimes \mathcal{O}_{A}$ for every fiber $A$. 
This shows that 
$\Omega_{\pi_{I}}^{k-l} \simeq \pi_{I}^{\ast}\pi_{I\ast}\Omega_{\pi_{I}}^{k-l}$ 
and so 
\begin{equation*}\label{eqn: Gr hol Leray}
{\rm Gr}^{l}_{L}\Omega_{\tilde{Y}_{I}}^{k} \simeq 
\pi_{I}^{\ast}\Omega_{{\DI}}^{l} \otimes \Omega_{\pi_{I}}^{k-l} \simeq 
\pi_{I}^{\ast}( \Omega_{{\DI}}^{l} \otimes \pi_{I \ast}\Omega_{\pi_{I}}^{k-l} ). 
\end{equation*}
This isomorphism is $W(I)_{{\R}}/U(I)_{{\R}}$-equivariant because holomorphic forms on $A$ are translation-invariant. 
Since the embedding 
${\E}_{{\la}}^{I}|_{\tilde{Y}_{I}}\hookrightarrow {\rm Gr}_{L}^{l}\Omega_{\tilde{Y}_{I}}^{k}$ 
is $W(I)_{{\R}}/U(I)_{{\R}}$-equivariant, 
it descends to 
${\E}_{{\la}'}\hookrightarrow \Omega_{{\DI}}^{l} \otimes \pi_{I \ast}\Omega_{\pi_{I}}^{k-l}$. 
\end{proof}

The homomorphisms in Lemma \ref{lemma: GrL} are equivariant with respect to the remaining group ${\GIhd}$. 
By taking the quotient by ${\GIhd}$, 
the assertions of Theorem \ref{thm: main} are now proved.

\subsection{Cuspidality}\label{ssec: cusp form}

Let $\omega$ be a holomorphic $k$-form on $X$ with $k=k(\alpha)$, $0<\alpha <g$. 
There are (at least) three different notions of ``restriction to boundary" for $\omega$: 
\begin{enumerate}
\item As a modular form, by the Siegel operators. 
\item As a differential form, by the restriction to the SNC boundary divisor of a smooth projective compactification. 
\item As a cohomology class, by the restriction to the boundary of the Borel-Serre compactification. 
\end{enumerate}
These correspond to different aspects of $\omega$, 
and are concerned with different types of compactification of $X$: 
Satake, toroidal and Borel-Serre respectively. 
Accordingly, we have three notions of ``vanishing at boundary''. 
As a consequence of Theorem \ref{thm: main}, we show that the first two are equivalent. 
(The equivalence with the third will be shown in Corollary \ref{cor: cuspidality Borel-Serre}.) 

\begin{corollary}\label{cor: cuspidality}
For each cusp $I$, we have $\Phi_{I}(\omega)=0$ if and only if 
the restriction of $\omega$ as a differential form to every irreducible component of $\pi^{-1}(X_{I})$ vanishes. 
In particular, the following conditions for $\omega$ are equivalent: 

(1) $\omega$ is a cusp form. 

(2) The restriction of $\omega$ as a differential form to every irreducible component of $X^{\Sigma}-X$ vanishes. 

(3) The restriction of $\omega$ as a differential form to every irreducible component of $X^{\Sigma}-X$ 
over a corank $1$ cusp vanishes. 

(4) If $X\hookrightarrow \bar{X}$ is an arbitrary smooth projective compactification with SNC boundary divisor, 
the restriction of $\omega$ as a differential form to every irreducible component of $\bar{X}-X$ vanishes. 
\end{corollary}

\begin{proof}
The equality \eqref{eqn: fundamental} implies 
$\Phi_{I}(\omega)=0 \Leftrightarrow \omega_{I}=0$ 
for a cusp $I$ of corank $r\leq \alpha$. 
By Corollary \ref{cor: restriction=pullback k-form}, 
$\omega_{I}=0$ is equivalent to the vanishing of the restriction of $\omega$ as a differential form  
to every irreducible component of $\pi^{-1}(X_{I})$. 
When $r>\alpha$, we have both 
$\Phi_{I}(\omega)=0$ (as ${\la}$ has corank $<r$) and 
$\omega_{I}=0$ (as $\omega_{I}$ has degree $>\dim Y_I$). 
This implies the first assertion. 

By moving $I$, we obtain $(1) \Leftrightarrow (2)$. 
Since whether a modular form is a cusp form can be detected only by the Siegel operators at corank $1$ cusps, 
we obtain $(1) \Leftrightarrow (3)$. 
The remaining implication $(2) \Rightarrow (4)$ is standard: 
after blow-up, we may assume that $\bar{X}$ dominates $X^{\Sigma}$ regularly.
\end{proof}

Note that if $I$ is a cusp of corank $1$, we have $\pi^{-1}(X_{I})=Y_{I}$. 
This is canonically determined and irreducible of codimension $1$. 
In this sense, the condition (3) is more canonical than (2), not just less redundant.

\section{Corank filtration and cohomology}\label{sec: corank}

In this section we give some applications of Corollary \ref{cor: cuspidality} 
to the cohomology of the Siegel modular variety $X = {\D}/{\G}$. 
We keep the setting and notation of \S \ref{sec: Omega}.

\subsection{Corank filtration}\label{ssec: corank I}

Let $0< \alpha < g$. 
Let $M_{{\la}}({\G})$ be the space of modular forms of weight ${\la}$ for ${\G}$. 
For each $0\leq r \leq \alpha$, we denote by $M_{{\la}}^{(r)}({\G})\subset M_{{\la}}({\G})$  
the subspace of those $\omega$ such that 
$\Phi_{I}(\omega) = 0$ at all cusps $I$ of corank $r+1$. 
This is equivalent to 
$\Phi_{I}(\omega)=0$ at all cusps $I$ of corank $>r$, 
because every cusp of corank $>r+1$ is in the closure of a cusp of corank $r+1$. 
Thus we have the \textit{corank filtration} 
\begin{equation}\label{eqn: corank filtration}
S_{{\la}}({\G}) = M_{{\la}}^{(0)}({\G}) \subset M_{{\la}}^{(1)}({\G}) \subset \cdots \subset M_{{\la}}^{(\alpha)}({\G}) = M_{{\la}}({\G}). 
\end{equation}
Here $M_{{\la}}^{(\alpha)}({\G}) = M_{{\la}}({\G})$ holds because ${\la}$ has corank $\alpha$. 
We also notice the following lower bound on the possible levels. 

\begin{lemma}\label{prop: corank filtration length}
Let $\alpha > g/2$. 
Then $M_{{\la}}^{(r)}({\G})=0$ for $r<2\alpha-g$. 
\end{lemma}

\begin{proof}
Let $r<\alpha$. 
If $\omega\in M_{{\la}}^{(r)}({\G})$ and $\dim I = r$, 
then $\Phi_{I}(\omega)$ is a cusp form on ${\DI}$ 
because cusps of ${\DI}$ are cusps of ${\D}$ of corank $>r$. 
By considering the Siegel operators on $M_{{\la}}^{(r)}({\G})$ at all cusps of corank $r$, 
we obtain the injective map 
\begin{equation*} 
(\Phi_{I})_{I} : 
M_{{\la}}^{(r)}({\G}) / M_{{\la}}^{(r-1)}({\G}) \hookrightarrow 
\bigoplus_{\dim I=r} S_{{\la}'}({\GIh}). 
\end{equation*}
By a theorem of Weissauer (\cite{We1} \S 5), 
${\la}'=((g+1)^{g-\alpha}, (g-\alpha)^{\alpha-r})$ is a singular weight in genus $g-r$ if and only if $2(g-\alpha)<g-r$, 
namely $r<2\alpha-g$. 
Since there are no singular cusp forms, we have $S_{{\la}'}({\GIh})=0$ in this case. 
Therefore we have 
\begin{equation*}
M_{{\la}}^{(r)}({\G}) = M_{{\la}}^{(r-1)}({\G}) = \cdots = M_{{\la}}^{(0)}({\G}) = 0 
\end{equation*}
as far as $r<2\alpha-g$. 
Here the last vanishing holds because 
${\la}$ itself is a singular weight in genus $g$ by our assumption $\alpha > g/2$. 
\end{proof}

\subsection{Corank filtration on cohomology}\label{ssec: corank II}

Let $0<k<\dim X$. 
The singular cohomology $H^k(X)=H^k(X, {\Q})$ of $X$ has a canonical mixed Hodge structure (\cite{De}, \cite{PS}). 
We denote by $(W_{\bullet}, F^{\bullet})$ the weight and Hodge filtrations. 
Since $X$ is smooth, $H^{k}(X)$ has weight $\geq k$. 
The pure weight $k$ part $W_kH^k(X)$ is the image of the natural map $H^k(X^{\Sigma})\to H^k(X)$. 
This is also the image of the natural map from the $L^{2}$-cohomology (\cite{HZ} Remark 5.5). 
The weight spectral sequence shows that the map $H^k(X^{\Sigma})\to H^k(X)$ is injective on 
$F^{k}H^{k}(X^{\Sigma})=H^{0}(\Omega_{X^{\Sigma}}^{k})$, so we have 
\begin{equation*}
F^{k}W_{k}H^{k}(X) \simeq H^{0}(\Omega_{X^{\Sigma}}^{k}) = H^{0}(\Omega_{X}^{k}).  
\end{equation*}
Therefore, by Weissauer's results \cite{We1}, 
we have $F^{k}W_{k}H^{k}(X)=0$ unless $k$ is of the form $k(\alpha)$ with $0<\alpha <g$, 
while when $k=k(\alpha)$, we have 
\begin{equation*}
F^{k(\alpha)}W_{k(\alpha)}H^{k(\alpha)}(X) \simeq M_{{\la}}({\G}). 
\end{equation*}
Thus the corank filtration \eqref{eqn: corank filtration} can be regarded as a filtration on 
$F^{k(\alpha)}W_{k(\alpha)}H^{k(\alpha)}(X)$. 
We show that this comes from a filtration on $W_{k(\alpha)}H^{k(\alpha)}(X)$ by sub ${\Q}$-Hodge structures. 

From now on, let $k=k(\alpha)$ with $0<\alpha <g$. 
For $0\leq r \leq g$ we consider the disjoint union  
\begin{equation*}
D[r] = \bigsqcup_{\substack{[\sigma]\in \Sigma_{I}^{\circ}/{\GIl} \\ \dim I>r}} \overline{\Delta([\sigma])}, 
\end{equation*}
where $[\sigma]$ ranges over ${\G}$-equivalence classes of all $I$-cones with $\dim I > r$, 
$\Delta([\sigma])$ is the $\sigma$-stratum in $X^{\Sigma}$ as defined in \S \ref{ssec: full toroidal}, 
and $\overline{\Delta([\sigma])}$ means the closure of $\Delta([\sigma])$ in $X^{\Sigma}$. 
When $[\sigma]$ is minimal in $\Sigma_{I}^{\circ}/{\GIl}$, 
the intersection $\overline{\Delta([\sigma])}\cap \pi^{-1}(X_{I})$ is 
the irreducible component $\overline{\Delta([\sigma])}^{I}$ of $\pi^{-1}(X_{I})$ defined in \eqref{eqn: irr comp D(I)}. 
By our SNC condition, each $\overline{\Delta([\sigma])}$ is smooth (see \cite{Ma2} \S 4). 
We have a natural finite map $D[r]\to X^{\Sigma}$. 

We define a subspace $V_{\Sigma}^{(r)}$ of $W_kH^k(X)$ as 
the image of 
$\ker (H^{k}(X^{\Sigma})\to H^{k}(D[r]))$ by 
$H^k(X^{\Sigma})\to W_kH^k(X)$. 
Since all these maps are morphisms of ${\Q}$-Hodge structures, 
$V_{\Sigma}^{(r)}$ is a sub ${\Q}$-Hodge structure of $W_kH^k(X)$. 
If $r<r'$, we have $D[r']\subset D[r]$, and so $V_{\Sigma}^{(r)} \subset V_{\Sigma}^{(r')}$. 
Therefore we obtain the filtration 
\begin{equation}\label{eqn: Hodge corank filtration}
V_{\Sigma}^{(0)} \subset V_{\Sigma}^{(1)} \subset \cdots  \subset V_{\Sigma}^{(g)} = W_kH^k(X) 
\end{equation}
on $W_kH^k(X)$ by sub ${\Q}$-Hodge structures.  

\begin{proposition}\label{prop: Hodge corank filtration}
Under the identification $F^{k}W_{k}H^{k}(X) = M_{{\la}}({\G})$, 
we have $F^{k}V_{\Sigma}^{(r)} = M_{{\la}}^{(r)}({\G})$. 
\end{proposition}

\begin{proof}
Taking the $F^k$-part in the definition of $V_{\Sigma}^{(r)}$, 
we see that $F^{k}V_{\Sigma}^{(r)}$ is the image of 
$\ker (H^0(\Omega_{X^{\Sigma}}^{k}) \to H^0(\Omega_{D[r]}^{k}))$ by 
$H^0(\Omega_{X^{\Sigma}}^{k}) \stackrel{\simeq}{\to} F^{k}W_{k}H^{k}(X)$. 
Here the map $H^0(\Omega_{X^{\Sigma}}^{k}) \to H^0(\Omega_{D[r]}^{k})$ 
is given by pullback of differential forms. 
A $k$-form $\omega$ is contained in its kernel if and only if the restriction of $\omega$ 
to every irreducible component of $\pi^{-1}(X_{I})$ vanishes for all cusps $I$ of corank $>r$. 
By Corollary \ref{cor: cuspidality}, 
this is equivalent to $\Phi_{I}(\omega)=0$ for all cusps $I$ of corank $>r$. 
This is exactly the condition $\omega\in M_{{\la}}^{(r)}({\G})$. 
\end{proof}

We show that the first step $V_{\Sigma}^{(0)}$ is a well-known subspace of $H^{k}(X)$. 
Let $H^{k}_{!}(X)$ be the \textit{interior cohomology}, 
namely the image of the natural map $H^{k}_{c}(X)\to H^{k}(X)$ 
from the compactly supported cohomology $H^{k}_{c}(X)$. 
Since $H^{k}_{c}(X)$ has weight $\leq k$, 
$H^{k}_{!}(X)$ is a sub ${\Q}$-Hodge structure of $W_{k}H^{k}(X)$. 

\begin{proposition}\label{prop: V1 = H!} 
We have $V^{(0)}_{\Sigma} = H^{k}_{!}(X)$. 
\end{proposition}

\begin{proof}
Let $D=X^{\Sigma}-X$ and $\tilde{D}$ be the normalization of $D$. 
Then $\tilde{D}$ is the disjoint union of $\overline{\Delta([\tau])}$ for all codimension $1$ strata $\Delta([\tau])$, 
namely $\dim \tau =1$. 
Since every stratum $\Delta([\sigma])$ is contained in the closure of some codimension $1$ stratum $\Delta([\tau])$, 
we have 
\begin{equation*}
\ker (H^{k}(X^{\Sigma})\to H^{k}(D[0])) = \ker (H^{k}(X^{\Sigma})\to H^{k}(\tilde{D})). 
\end{equation*}
Hence $V_{\Sigma}^{(0)}$ is the image of 
$\ker (H^k(X^{\Sigma})\to H^k(\tilde{D}))$ by $H^k(X^{\Sigma})\to H^k(X)$. 

The map $H^k(X^{\Sigma})\to H^k(\tilde{D})$ factors as 
\begin{equation*}
H^k(X^{\Sigma}) \stackrel{j^{\ast}}{\to} H^k(D) \stackrel{\nu^{\ast}}{\to} H^k(\tilde{D}) 
\end{equation*}
where $j\colon D\hookrightarrow X^{\Sigma}$ is the inclusion map and 
$\nu \colon \tilde{D}\to D$ is the normalization map. 
By \cite{PS} Corollary 5.42, we have $\ker(\nu^{\ast}) =W_{k-1}H^{k}(D)$. 
Since $H^{k}(X^{\Sigma})$ is pure of weight $k$, 
we find that 
\begin{equation*}
j^{\ast}H^{k}(X^{\Sigma})\cap \ker(\nu^{\ast}) = j^{\ast}H^{k}(X^{\Sigma})\cap W_{k-1}H^{k}(D) = \{ 0 \} 
\end{equation*} 
by the strictness of morphisms of mixed Hodge structures. 
Therefore $\nu^{\ast}$ is injective on ${\rm Im}(j^{\ast})$, and so 
$\ker(\nu^{\ast}\circ j^{\ast}) = \ker(j^{\ast})$. 
This shows that $V_{\Sigma}^{(0)}$ is the image of $\ker(j^{\ast})$ by $H^k(X^{\Sigma})\to H^k(X)$. 

Finally, the cohomology exact sequence for the pair $(X^{\Sigma}, D)$ shows that  
$\ker (j^{\ast})$ is the image of  $H^{k}(X^{\Sigma}, D) \to H^{k}(X^{\Sigma})$. 
Hence $V_{\Sigma}^{(0)}$ is the image of the natural map $H^{k}(X^{\Sigma}, D) \to H^k(X)$. 
Since $H^{k}(X^{\Sigma}, D)=H^{k}_{c}(X)$ (see \cite{PS} Corollary B.14), 
we obtain $V_{\Sigma}^{(0)} = H^{k}_{!}(X)$. 
\end{proof}

Thus we find that the cohomological corank filtration $V_{\Sigma}^{(\bullet)}$ interpolates $H^{k}_{!}(X)$ and $W_{k}H^{k}(X)$. 
If we combine Propositions \ref{prop: Hodge corank filtration} and \ref{prop: V1 = H!}, 
we obtain 

\begin{corollary}\label{cor: cusp form Hodge}
$F^{k}H^{k}_{!}(X)\simeq S_{{\la}}({\G})$. 
\end{corollary}

Since $S_{{\la}}({\G})\ne M_{{\la}}({\G})$ in general (see, e.g., \cite{We2}, \cite{SM}), 
this shows that $H^{k}_{!}(X)\ne W_kH^k(X)$ in general.  
 
Let $X^{bs}$ be the Borel-Serre compactification of $X$ (\cite{BS}) 
and $\partial X^{bs}$ be its boundary. 
It is well-known that 
$H^{k}_{!}(X)$ is the kernel of the restriction map 
$H^{k}(X)\simeq H^{k}(X^{bs})\to H^{k}(\partial X^{bs})$. 
Hence Corollary \ref{cor: cusp form Hodge} implies the following. 
 
\begin{corollary}\label{cor: cuspidality Borel-Serre}
A holomorphic $k$-form $\omega$ is a cusp form if and only if 
the restriction of $\omega$ as a cohomology class to $\partial X^{bs}$ vanishes. 
\end{corollary}
 
We conclude our discussion of the filtration \eqref{eqn: Hodge corank filtration} 
with a remark on the dependence on $\Sigma$. 
By Proposition \ref{prop: V1 = H!}, the first step $V_{\Sigma}^{(0)}$ does not depend on $\Sigma$. 
However, at least a priori, the higher steps $V_{\Sigma}^{(r)}$ depend on $\Sigma$. 
To get rid of this, we can take limit as follows. 
If $\Sigma'$ is an SNC subdivision of $\Sigma$, 
we have a morphism $X^{\Sigma'}\to X^{\Sigma}$, 
and every boundary stratum of $X^{\Sigma'}$ is mapped to 
some boundary stratum of $X^{\Sigma}$ over the same cusp. 
This implies 
$V_{\Sigma}^{(r)} \subset V_{\Sigma'}^{(r)}$. 
This enables us to take the inductive limit 
\begin{equation*}
V^{(r)} := \varinjlim_{\Sigma} V_{\Sigma}^{(r)} = \bigcup_{\Sigma} V_{\Sigma}^{(r)}. 
\end{equation*}
This limit terminates by the finite-dimensionality of $W_{k}H^{k}(X)$. 
By Proposition \ref{prop: Hodge corank filtration}, 
the $F^{k}$-part is unchanged: 
\begin{equation*}
F^{k}V^{(r)} = M_{{\la}}^{(r)}({\G}). 
\end{equation*}
In this way, we can obtain a filtration on $W_kH^k(X)$ by ${\Q}$-Hodge structures 
which does not depend on $\Sigma$ and extends the corank filtration.

\subsection{Cusp forms and a Hodge component}\label{ssec: Hodge vanish}

Let us write $n=\dim X = g(g+1)/2$. 
We again allow $k$ to be an arbitrary value with $0< k<n$. 
In this subsection we study a Hodge component of $H^{2n-k}(X)$. 
By Borel-Serre \cite{BS}, we have $H^{2n-k}(X)=0$ if $k<g$. 
In what follows, we assume $g\leq k <n$. 
Then $H^{2n-k}(X)$ has weight $\geq 2n-k$, 
and the Hodge filtration on the pure part $W_{2n-k}H^{2n-k}(X)$ has levels in $[n-k, n]$ (see \cite{PS} Theorem 5.39). 
For a ${\C}$-linear space $V$ we denote by $V^{\ast}$ the complex conjugate of $V^{\vee}$. 

\begin{proposition}\label{prop: Hodge vanish}
Let $g\leq k <n$. 
We have 
\begin{equation*}
F^{n}W_{2n-k}H^{2n-k}(X) \simeq 
\begin{cases}
S_{{\la}}({\G})^{\ast} & \; k=k(\alpha), \; 0<\alpha \leq g/2 \\ 
\; 0 & \; \textrm{otherwise} 
\end{cases}
\end{equation*}
\end{proposition}

\begin{proof}
By Poincare duality (see \cite{PS} \S 6.3), we have an isomorphism of pure Hodge structures 
\begin{equation*}
W_{2n-k}H^{2n-k}(X) \simeq {\rm Gr}^{W}_{k}H^{k}_{c}(X)^{\vee}(-n) 
\end{equation*}
where $(-n)$ means the Tate shift. 
Therefore we have 
\begin{equation*}
F^{n}W_{2n-k}H^{2n-k}(X) \simeq {\rm Gr}_{F}^{0}{\rm Gr}^{W}_{k}H^{k}_{c}(X)^{\vee} 
\simeq F^{k}{\rm Gr}^{W}_{k}H^{k}_{c}(X)^{\ast} \simeq F^{k}H^{k}_{c}(X)^{\ast}. 
\end{equation*}
Here the last isomorphism holds because the lower weight graded quotients 
${\rm Gr}^{W}_{l}H^{k}_{c}(X)$, $l<k$, have Hodge level $<k$. 

As before, we take a smooth projective toroidal compactification $X^{\Sigma}$ with SNC boundary divisor $D=X^{\Sigma}-X$. 
For the mixed Hodge structure on $H^{k}_{c}(X)$, 
it is well-known that we have the Hodge spectral sequence of the form 
\begin{equation*}
E_{1}^{p,q} = H^{q}(X^{\Sigma}, \Omega_{X^{\Sigma}}^{p}(\log D)(-D)) 
\quad \Rightarrow \quad 
E_{\infty}^{p+q} = H^{p+q}_{c}(X) 
\end{equation*}
which degenerates at $E_1$. 
(See, e.g., 
\cite{Fu} \S 2.2 for a proof.) 
Therefore 
\begin{equation}\label{eqn: FkHkc}
F^{k}H^{k}_{c}(X) \simeq H^{0}(X^{\Sigma}, \Omega_{X^{\Sigma}}^{k}(\log D)(-D)). 
\end{equation}
By the theorem of Freitag-Pommerening \cite{FP}, we have 
\begin{equation*}
H^{0}(\Omega_{X^{\Sigma}}^{k}(\log D)(-D)) \subset 
H^{0}(\Omega_{X^{\Sigma}}^{k}(\log D)) = H^{0}(\Omega_{X^{\Sigma}}^{k}). 
\end{equation*}
When $k$ is not of the form $k(\alpha)$, we have 
$H^{0}(\Omega_{X^{\Sigma}}^{k})=0$ by \cite{We1}. 

Next we consider the case $k=k(\alpha)$. 

\begin{claim}\label{claim: log D -D}
Let $\omega$ be a holomorphic $k$-form on $X^{\Sigma}$. 
Then $\omega$ takes values in $\Omega_{X^{\Sigma}}^{k}(\log D)(-D)$ if and only if 
the restriction of $\omega$ as a differential form to 
every irreducible component of $D$ vanishes. 
\end{claim}

\begin{proof}
This is a local problem, and it suffices to prove this equivalence at a general point $x$ of every irreducible component of $D$. 
Let $z_1, \cdots, z_n$ be local coordinates around $x$ such that $D$ is locally defined by $z_1=0$. 
We express $\omega$ locally as 
\begin{equation*}
\omega = \sum_{I} a_{I} dz_{1}\wedge dz_{I} + \sum_{J}b_{J}dz_{J} 
\end{equation*}
where $I, J \subset \{ 2, \cdots, n \}$ 
and $a_{I}$, $b_{J}$ are holomorphic functions around $x$. 
The $I$-components 
$a_{I}dz_{1}\wedge dz_{I} = z_{1}a_{I}(dz_{1}/z_{1})\wedge dz_{I}$ 
already take values in $\Omega_{X^{\Sigma}}^{k}(\log D)(-D)$. 
This implies that $\omega$ takes values in $\Omega_{X^{\Sigma}}^{k}(\log D)(-D)$ if and only if 
$b_{J}$ vanishes at $D$ for every $J$. 
Since the restriction $\omega|_{D}$ of $\omega$ to $D$ as a differential form is given by 
$\sum_{J}(b_{J}|_{D})dz_{J}$, this condition is equivalent to $\omega|_{D} \equiv 0$.  
\end{proof}
 
If we combine Claim \ref{claim: log D -D}, \eqref{eqn: FkHkc} and Corollary \ref{cor: cuspidality}, 
we obtain 
\begin{equation*}
F^{k(\alpha)}H^{k(\alpha)}_{c}(X) \simeq S_{{\la}}({\G}). 
\end{equation*}
By Lemma \ref{prop: corank filtration length}, we have $S_{{\la}}({\G})=0$ if $\alpha > g/2$. 
This finishes the proof of Proposition \ref{prop: Hodge vanish}. 
\end{proof}

\begin{remark}
Even when ${\G}$ is not neat, 
$H^{\ast}({\D}/{\G})$ still has a canonical mixed Hodge structure. 
This can be identified with the ${\G}/{\G}'$-invariant part of the mixed Hodge structure on $H^{\ast}({\D}/{\G}')$ 
for a neat subgroup ${\G}'\lhd {\G}$ (see \cite{Ma2} \S 7). 
Hence Proposition \ref{prop: Hodge vanish} is still valid in the non-neat case. 
\end{remark}

\section{The orthogonal case}\label{sec: orthogonal}

In this section we prove an analogue of Theorem \ref{thm: main} for orthogonal modular varieties. 
This is simpler than the Siegel case mainly because only $1$-dimensional cusps are involved 
and the toroidal compactification over them is unique. 
Actually this is where we started. 
Since the basic story is parallel to the Siegel case, we keep the exposition brief. 
We refer to \cite{Ma1} for a basic theory of vector-valued orthogonal modular forms.

\subsection{Orthogonal modular varieties}

Let $V$ be a ${\Q}$-linear space equipped with a rational quadratic form $( \cdot, \cdot )$ of signature $(2, n)$. 
The open set of the isotropic quadratic in ${\proj}V_{{\C}}$ defined by the condition $(v, \bar{v})>0$ consists of two connected components.  
Let ${\D}$ be one of them. 
Let ${\G}$ be a neat arithmetic subgroup of ${\rm O}(V)$. 
Then $X={\D}/{\G}$ has the structure of a smooth quasi-projective variety of dimension $n$, 
called an \textit{orthogonal modular variety}. 

The Baily-Borel compactification $X^{\ast}={\D}^{\ast}/{\G}$ is obtained by adjoining 
$0$-dimensional and $1$-dimensional cusps to ${\D}$, 
which correspond to isotropic lines $I$ and isotropic planes $J$ in $V$ respectively. 
An isotropic line $I$ determines the boundary point $[I_{{\C}}]\in {\proj}V_{{\C}}$ of ${\D}$. 
We denote by $p_{I}$ its image in $X^{\ast}$. 
An isotropic plane $J$ determines the boundary component $\mathbb{H}_{J}$ of ${\D}$ as 
the upper half plane in ${\proj}J_{{\C}} \subset {\proj}V_{{\C}}$. 
Dividing $\mathbb{H}_{J}$ by the stabilizer of $J$ in ${\G}$, 
we obtain the $J$-cusp $X_J$ in $X^{\ast}$ as a modular curve. 

We take a smooth projective toroidal compactification $X^{\Sigma}$ of $X$ with SNC boundary divisor. 
Here $\Sigma=(\Sigma_{I})_{I}$ is a collection of fans, 
one for each ${\G}$-equivalence class of $0$-dimensional cusps $I$. 
No choice is required for $1$-dimensional cusps $J$: 
it is canonical because $\dim U(J)=1$. 
Let $\pi\colon X^{\Sigma}\to X^{\ast}$ be the projection. 
For a $0$-dimensional cusp $I$, 
$\pi^{-1}(p_{I})$ is a union of smooth projective toric varieties. 
For a $1$-dimensional cusp $J$, 
$\pi^{-1}(X_J)$ is irreducible and of codimension $1$ in $X^{\Sigma}$; 
it is an abelian fibration over $X_J$. 
We denote $Y_{J}=\pi^{-1}(X_{J})$ and write $\pi_{J}\colon Y_{J}\to X_{J}$ for the projection.

\subsection{Holomorphic differential forms}

Our purpose is to study the Siegel operators for holomorphic differential forms on $X$. 
When $0<k<n/2$, there is no holomorphic $k$-form on $X$ (see \cite{Ma1} Chapter 9). 
In what follows, we assume $n/2 \leq k <n$. 
By the extension theorem of Pommerening \cite{Po}, 
every holomorphic $k$-form $\omega$ on $X$ extends holomorphically over $X^{\Sigma}$. 

In the orthogonal case, $\Omega_{{\D}}^{k}$ is already irreducible as an automorphic vector bundle. 
The Siegel operator at a $0$-dimensional cusp $I$ is trivial (\cite{Ma1} Proposition 3.8). 
On the other hand, the restriction of $\omega$ as a differential form to the irreducible components of $\pi^{-1}(p_{I})$ vanishes 
because they are rational varieties. 
Thus, from either point of view, 
only $1$-dimensional cusps are interesting as regards to the boundary behavior of holomorphic differential forms.  

Let $J$ be an isotropic plane in $V$. 
We write $V(J)=J^{\perp}/J$. 
The Siegel operator at $J$ for $\omega$ produces 
a $\wedge^{k-1}V(J)_{{\C}}$-valued cusp form $\Phi_{J}(\omega)$ of weight $k+1$ on the modular curve $X_{J}$ (\cite{Ma1} Chapter 6). 
Let $\mathcal{L}_{J}$ be the Hodge line bundle on $X_{J}$. 
We denote by $\omega|_{Y_{J}}$ the restriction of $\omega$ to $Y_{J}\subset X^{\Sigma}$ as a differential form. 
We can now state the result of this section. 

\begin{theorem}\label{thm: orthogonal}
The $k$-form $\omega|_{Y_{J}}$ takes values in the sub vector bundle 
\begin{equation}\label{eqn: sub VB orthogonal}
\pi_{J}^{\ast}\Omega^{1}_{X_{J}} \wedge \Omega^{k-1}_{Y_{J}} \simeq 
\pi_{J}^{\ast}(\Omega^{1}_{X_{J}} \otimes \pi_{J \ast}\Omega^{k-1}_{\pi_{J}}) \simeq 
\pi_{J}^{\ast}(\wedge^{k-1}V(J)_{{\C}} \otimes \mathcal{L}_{J}^{\otimes k+1}) 
\end{equation}
of $\Omega^{k}_{Y_{J}}$, and we have 
\begin{equation}\label{eqn: restriction=Siegel orthogonal}
\omega|_{Y_{J}} = \pi_{J}^{\ast} \Phi_{J}(\omega) 
\end{equation}
as sections of \eqref{eqn: sub VB orthogonal}. 
\end{theorem}

\begin{proof}
We first recall the structure of $\pi_{J}$ (cf.~\cite{Ma1} Chapter 5) 
and verify the isomorphism \eqref{eqn: sub VB orthogonal}. 
We choose a line $I\subset J$ and put $U(I)=(I^{\perp}/I)\otimes I$. 
The projection from the point $[I_{{\C}}]\in {\proj}V_{{\C}}$ realizes ${\D}$ as the tube domain in $U(I)_{{\C}}$. 
Then the linear projections $I^{\perp}/I \to I^{\perp}/J \to I^{\perp}/J^{\perp}$ induce a two-step fibration 
\begin{equation}\label{eqn: Siegel domain orthogonal}
{\D} \stackrel{\pi_1}{\to} \mathcal{V}_{J} \stackrel{\pi_2}{\to} \mathbb{H}_{J}, 
\end{equation}
where $\pi_1$ is a fibration of upper half planes and 
$\pi_2$ is an affine space bundle isomorphic to $V(J)_{{\C}}\otimes \mathcal{O}(1)$. 
This is the Siegel domain realization of ${\D}$ with respect to $J$. 
It is in fact independent of the choice of $I$. 

Let ${\G}(J)$ be the stabilizer of $J$ in ${\G}$. 
Dividing the $\pi_{1}$-fibers by a rank $1$ abelian subgroup $U(J)_{{\Z}}$ of ${\G}(J)$, 
we get a punctured-disc fibration $\mathcal{X}(J)\to \mathcal{V}_{J}$. 
The partial toroidal compactification $\overline{\mathcal{X}(J)}$ of $\mathcal{X}(J)$ is obtained 
by filling the origins of the punctured discs. 
The boundary divisor is identified with $\mathcal{V}_{J}$. 
Then $\pi_{J}\colon Y_J \to X_J$ is the quotient of 
$\pi_2 \colon \mathcal{V}_{J} \to \mathbb{H}_{J}$ by ${\G}(J)$. 
The $\pi_{J}$-fibers are quotients of the $\pi_{2}$-fibers by the translation by a lattice. 
This shows that 
\begin{equation*}
\Omega_{\pi_{J}}^{1}\simeq \pi_{J}^{\ast}\pi_{J \ast}\Omega_{\pi_{J}}^{1} 
\simeq \pi_{J}^{\ast} (V(J)_{{\C}}\otimes \mathcal{L}_{J}). 
\end{equation*}
Taking the exterior power and using $\Omega_{X_{J}}^{1} \simeq \mathcal{L}_{J}^{\otimes 2}$, 
we obtain \eqref{eqn: sub VB orthogonal}. 

The Siegel domain realization \eqref{eqn: Siegel domain orthogonal} defines the filtration 
\begin{equation}\label{eqn: filtration Omega1 orthogonal}
\pi_{1}^{\ast}\pi_{2}^{\ast}\Omega_{\mathbb{H}_{J}}^{1} \subset \pi_{1}^{\ast}\Omega_{\mathcal{V}_{J}}^{1} \subset \Omega_{{\D}}^{1} 
\end{equation} 
on $\Omega_{{\D}}^{1}$.  
The description of \eqref{eqn: Siegel domain orthogonal} as linear projections shows that 
via the $I$-trivialization 
$\Omega_{{\D}}^{1}\simeq U(I)_{{\C}}^{\vee}\otimes \mathcal{O}_{{\D}}$, 
the filtration \eqref{eqn: filtration Omega1 orthogonal} corresponds to the constant filtration 
$J/I \subset J^{\perp}/I \subset I^{\perp}/I$ on $I^{\perp}/I=(I^{\perp}/I)^{\vee}$ 
twisted by $I^{\vee}$. 
This implies that \eqref{eqn: filtration Omega1 orthogonal} coincides with 
the ``$J$-filtration'' considered in \cite{Ma1} Chapter 8. 
This induces the filtration 
\begin{equation}\label{eqn: filtration Omegal orthogonal}
\pi_{1}^{\ast}(\pi_{2}^{\ast}\Omega_{\mathbb{H}_{J}}^{1}\wedge \Omega_{\mathcal{V}_{J}}^{k-1}) 
\subset \pi_{1}^{\ast}\Omega_{\mathcal{V}_{J}}^{k} \subset \Omega_{{\D}}^{k} 
\end{equation}
on $\Omega_{{\D}}^{k}$. 
The sub vector bundle 
$\pi_{1}^{\ast}(\pi_{2}^{\ast}\Omega_{\mathbb{H}_{J}}^{1}\wedge \Omega_{\mathcal{V}_{J}}^{k-1})$ 
coincides with last step of the $J$-filtration on $\Omega_{{\D}}^{k}$ (cf.~\cite{Ma1} Example 8.8). 

We pass to the canonical extension 
$\Omega^{k}_{\overline{\mathcal{X}(J)}}(\log \mathcal{V}_{J})$ 
of $\Omega^{k}_{\mathcal{X}(J)}$ 
and take the restriction to the boundary divisor $\mathcal{V}_{J}$. 
Then \eqref{eqn: filtration Omegal orthogonal} induces the filtration 
\begin{equation*}
\pi_{2}^{\ast}\Omega_{\mathbb{H}_{J}}^{1}\wedge \Omega_{\mathcal{V}_{J}}^{k-1} 
\subset \Omega_{\mathcal{V}_{J}}^{k} 
\subset \Omega^{k}_{\overline{\mathcal{X}(J)}}(\log \mathcal{V}_{J})|_{\mathcal{V}_{J}}. 
\end{equation*}
Here the second inclusion is obtained by pulling back $k$-forms by $\overline{\mathcal{X}(J)}\to \mathcal{V}_{J}$ 
and restricting them again to $\mathcal{V}_{J}$ as sections of $\Omega_{\overline{\mathcal{X}(J)}}^{k}$. 
Restriction of $k$-forms to $\mathcal{V}_{J}$ as differential forms is given by 
a projection from 
$\Omega^{k}_{\overline{\mathcal{X}(J)}}|_{\mathcal{V}_{J}}$ 
to this \textit{sub} vector bundle. 

Now, by the theory of Siegel operators in \cite{Ma1}, 
the restriction of $\omega$ to the boundary divisor $\mathcal{V}_{J}$ of $\overline{\mathcal{X}(J)}$ 
as a section of the canonical extension 
$\Omega^{k}_{\overline{\mathcal{X}(J)}}(\log \mathcal{V}_{J})$ 
takes values in the last step of the $J$-filtration. 
As noted above, this is 
$\pi_{2}^{\ast}\Omega_{\mathbb{H}_{J}}^{1}\wedge \Omega_{\mathcal{V}_{J}}^{k-1}$. 
Since it is contained in $\Omega_{\mathcal{V}_{J}}^{k}$, 
we find that $\omega$ extends holomorphically over $\overline{\mathcal{X}(J)}$ (as guaranteed by the theorem of Pommerening) 
and that the restriction of $\omega$ to $\mathcal{V}_{J}$ as a differential form already coincides with the restriction as a section of 
$\Omega^{k}_{\overline{\mathcal{X}(J)}}$. 
We can denote it unambiguously by $\omega|_{\mathcal{V}_{J}}$. 
Then we have 
$\omega|_{\mathcal{V}_{J}} = \pi_{2}^{\ast}\Phi_{J}(\omega)$ 
by \cite{Ma1} Theorem 6.1. 
\end{proof}

As in the Siegel case, the equality \eqref{eqn: restriction=Siegel orthogonal} ensures that 
Corollary \ref{cor: intro} holds also in the orthogonal case.

\begin{remark}\label{remark: Pommerening orthogonal}
The above argument contains a direct proof of 
the Pommerening extension theorem over $1$-dimensional cusps. 
We can also give a direct proof over $0$-dimensional cusps as follows. 
We prove vanishing of the residues along the toroidal boundary strata inductively. 
The starting case is codimension $k$ strata. 
The residues there are components of the constant term of the Fourier expansion of $\omega$, 
which vanishes by \cite{Ma1} Proposition 3.8. 
Then induction proceeds because the closure of a stratum is a smooth projective toric variety 
which has no holomorphic forms of positive degree. 
\end{remark}


\end{document}